\documentclass[a4paper, reqno, 11pt]{amsart}
\usepackage{amsmath}
\usepackage{amssymb}
\usepackage{cases}
\usepackage[all]{xy}
\usepackage{booktabs}
\allowdisplaybreaks[4]
\usepackage{enumerate}
\usepackage{array, longtable}
\usepackage{caption}

\newtheorem{thm}{Theorem}[section]
\newtheorem{prop}[thm]{Proposition}

\newtheorem{lem}[thm]{Lemma}

\newtheorem{cor}[thm]{Corollary}

\newtheorem{notation}[thm]{Notation}

\theoremstyle{definition}
\newtheorem{definition}[thm]{Definition}
\newtheorem{example}[thm]{Example}

\theoremstyle{remark}
\newtheorem{remark}[thm]{Remark}

\numberwithin{equation}{section}

\newcommand{\bQ}{\mathbb{Q}}

\newcommand{\bP}{\mathbb{P}}

\newcommand\OO{{\mathcal{O}}}

\newcommand{\roundup}[1]{\lceil{#1}\rceil}

\newcommand\lcm{{\text{lcm}}}

\newcommand{\Nklt}{\operatorname{Nklt}}

\newcommand{\Mov}{\operatorname{Mov}}

\newcolumntype{C}{>{$}c<{$}}
\newcolumntype{L}{>{$}l<{$}}

\begin{document}

\title{On the anti-canonical geometry of weak $\mathbb{Q}$-Fano threefolds, III}
\date{\today}

\author{Chen Jiang}
\address{Chen Jiang, Shanghai Center for Mathematical Sciences \& School of Mathematical Sciences, Fudan University, Shanghai, 200438, China}
\email{chenjiang@fudan.edu.cn}

\author{Yu Zou}
\address{Yu Zou, Yau Mathematical Sciences Center, Tsinghua University, Beijing, 100084, China}
\email{fishlinazy@tsinghua.edu.cn}

\maketitle
\pagestyle{myheadings} \markboth{\hfill C.~Jiang and Y.~Zou
\hfill}{\hfill The anti-canonical geometry of Fano $3$-folds\hfill}

\begin{abstract}
For a terminal weak $\bQ$-Fano $3$-fold $X$, we show that the $m$-th anti-canonical map defined by $|-mK_X|$ is birational for all $m\geq59$.
\end{abstract}

\section{Introduction}
Throughout this paper, we work over an algebraically closed field of characteristic $0$ (for instance, the complex number field $\mathbb{C}$). We adopt standard notation in \cite{KM}.

A normal projective variety $X$ is called a {\it weak $\bQ$-Fano variety} (resp. {\it $\bQ$-Fano variety}) if $-K_X$ is nef and big (resp. ample). 
According to the minimal model program, (weak) $\bQ$-Fano varieties form a fundamental class in birational geometry.
Motivated by the classification theory of $3$-dimensional algebraic varieties, we are interested in the study of explicit geometry of (weak) $\bQ$-Fano varieties with terminal or canonical singularities. In this direction there are a lot of works in the literature, see for example
 \cite{Prok05, Prok07, CC08, Prok10, Chen11, Prok13, CJ16, CJ20, JZ21, Jiang21, Jiang22}.

Given a terminal weak $\mathbb{Q}$-Fano $3$-fold $X$, the {\it $m$-th anti-canonical map} $\varphi_{-m, X}$ (or simply $\varphi_{-m}$) is the rational map induced by the linear system $|-mK_X|$.
We are interested in the fundamental question of finding an optimal integer $c_3$ such that $\varphi_{-m}$ is birational for all $m\geq c_3$. The existence of such $c_3$ follows from the boundedness result in \cite{KMMT}. More generally, Birkar \cite{bir19} showed that, fix a positive integer $d$, there exists a positive integer $c_d$ such that 
$\varphi_{-m}$ is birational for all $m\geq c_d$ and for all terminal weak $\mathbb{Q}$-Fano $d$-folds, which is one important step towards the solution of the Borisov--Alexeev--Borisov conjecture.
The following example shows that $c_3\geq 33$.

\begin{example}[{\cite[List~16.6, No.~95]{IF00}}]\label{ex33}
A general weighted hypersurface $X_{66}\subset\mathbb{P}(1, 5, 6,$ $22, 33)$ is a $\mathbb{Q}$-factorial terminal $\mathbb{Q}$-Fano $3$-fold of Picard number $1$ with $\varphi_{-m}$ birational for $m\geq 33$ but $\varphi_{-32}$ not birational. 
\end{example}

In \cite{CJ16}, it was showed that for a terminal weak $\mathbb{Q}$-Fano $3$-fold $X$, $\varphi_{-m}$ is birational for all $m\geq 97$, which seems far from being optimal comparing to Example~\ref{ex33}. Later in \cite{CJ20}, it was showed that any terminal weak $\mathbb{Q}$-Fano $3$-fold is birational to some terminal weak $\mathbb{Q}$-Fano $3$-fold $Y$ such that $\varphi_{-m, Y}$ is birational for all $m\geq 52$.
Moreover, in recent works \cite{Jiang21, Jiang22}, we can make use of the behavior of the pluri-anti-canonical maps studied in \cite{CJ16} in the classification of terminal $\mathbb{Q}$-Fano $3$-folds.
So we believe that a better understanding of the behavior of the pluri-anti-canonical maps (including new methods developed during the approach) will help us understand the classification of terminal $\mathbb{Q}$-Fano $3$-folds better.

The main goal of this paper is to give an improvement of \cite{CJ16, CJ20} without passing to a birational model. The main theorem of this paper is the following.

\begin{thm}\label{thm1}
Let $X$ be a terminal weak $\bQ$-Fano $3$-fold. Then the $m$-th anti-canonical map $\varphi_{-m}$ defined by $|-mK_X|$ is birational for all $m\geq59$.
\end{thm}

\begin{remark}
Theorem~\ref{thm1} holds for canonical weak $\bQ$-Fano $3$-folds by taking a $\bQ$-factorial terminalization by \cite[Theorem~6.23, Theorem~6.25]{KM}. 
\end{remark}

For terminal $\bQ$-Fano $3$-folds, we have a slightly better bound.
\begin{thm}\label{thm2}
Let $X$ be a terminal $\bQ$-Fano $3$-fold. Then the $m$-th anti-canonical map $\varphi_{-m}$ defined by $|-mK_X|$ is birational for all $m\geq58$.
\end{thm}

To prove the main theorem, we already have several criteria to determine the birationality in \cite{CJ16, CJ20}, which are optimal in many cases (cf. \cite[Example~5.12]{CJ16}).
In order to study the birationality of $|-mK_X|$, 
as indicated in \cite{Chen11, CJ16, CJ20}, it is crucial to study when 
$|-mK_X|$ is not composed with a pencil. In fact, finding a criterion for 
$|-mK_X|$ not composed with a pencil is one of the central problems in \cite{CJ16, CJ20} (\cite[Problem~1.3]{CJ16}, \cite[Problem~1.5]{CJ20}). Comparing to the birationality criteria, the non-pencil criteria in \cite{CJ16, CJ20} are not satisfactory.
As one of the main ingredients of this paper, we give a new criterion for 
$|-mK_X|$ not composed with a pencil.

\begin{thm}[{=Theorem~\ref{npthm}}]
Let $X$ be a terminal weak $\bQ$-Fano $3$-fold. If 
$$
h^0(X,-mK_X)>12m+1
$$
for some positive integer $m$, then $|-mK_X|$ is not composed with a pencil.
\end{thm}

The following special case is already interesting for the study of anti-canonical systems of terminal weak $\bQ$-Fano $3$-folds, and might have applications on upper bounds of degrees of terminal weak $\bQ$-Fano $3$-folds (cf. \cite{Prok05, Prok07}).

\begin{cor}
Let $X$ be a terminal weak $\bQ$-Fano $3$-fold. If 
$h^0(X,-K_X)>13$, then $|-K_X|$ is not composed with a pencil.
\end{cor}

The paper is organized as the following.
In Section~\ref{sec 2}, we recall basic knowledge.
In Section~\ref{sec 3}, we recall the birationality criteria of terminal weak $\bQ$-Fano $3$-folds in \cite{CJ16, CJ20} with some generalizations.
In Section~\ref{sec 4}, we prove the new criterion Theorem~\ref{npthm} and give an effective method to apply it.
In Section~\ref{sec 5}, we prove the main theorem.

\section*{\bf Notation}

For the convenience of readers, we list here the notation that will be frequently used in this paper. Let $X$ be a terminal weak $\bQ$-Fano $3$-fold. 

{\footnotesize
\begin{longtable*}{LL}

\hline
\endfirsthead

\hline 

\endhead
\hline
\endfoot

\hline \hline
\endlastfoot
\varphi_{-m} & \text{the rational map defined by } |-mK_X|\\
P_{-m}=h^0(X, -mK_X)& \text{the $m$-th anti-plurigenus of 
$X$}\\
B_X=\{(b_i, r_i)\}& \text{the Reid basket of orbifold points of $X$}\\
\mathcal{R}_X=\{r_i\}& \text{the collection of local indices of $X$}\\
r_X=\lcm\{r_i\mid r_i\in \mathcal{R}_X\}& \text{the Cartier index of $K_X$}\\
r_{\max}=\max\{r_i\mid r_i\in\mathcal{R}_X\}&\text{the maximal local index of $X$}\\

\sigma(B_X)=\sum_i b_i& \text{an invariant of $B_X$ contributing to the Riemann-Roch formula}\\

\sigma'(B_X)=\sum_i\frac{b_i^2}{r_i}&\text{an invariant of $B_X$ contributing to the Riemann-Roch formula} \\

\gamma(B_X)=\sum_{i} \frac{1}{r_i}-\sum_{i} r_i+24&\text{an invariant of $B_X$ from the Miyaoka inequality} \\

B^{(0)}_X=\{n_{1, r}^0\times (1, r)\}& \text{the initial basket of $B_X$}\\
\hline
\end{longtable*}
}

\section{Preliminaries}\label{sec 2}

Let $X$ be a terminal weak $\bQ$-Fano $3$-fold. Denote by $r_X$ the Cartier index of $K_X$. For any positive integer $m$, the number $P_{-m}=h^0(X, \OO_X(-mK_X))$ is called the {\it $m$-th anti-plurigenus} of $X$ and $\varphi_{-m}$ denotes the {\it $m$-th anti-canonical map} defined by $|-mK_X|$.

\subsection{The fibration induced by $|D|$}\label{r-m}
Let $X$ be a terminal weak $\bQ$-Fano $3$-fold. Consider a $\bQ$-Cartier Weil divisor $D$ on $X$ with $h^0(X, D)\geq 2$. Then there is a rational map defined by $|D|$:
$$
\Phi_{|D|}: X \dashrightarrow \mathbb{P}^{h^0(X, D)-1}.
$$
By Hironaka's desingularization theorem, we can take a projective birational morphism $\pi: W \rightarrow X$ such that:

\begin{enumerate}[(i)]
 \item $W$ is smooth;

\item the movable part $|M|$ of $|\lfloor \pi^*(D)\rfloor|$ is base point free and, consequently, $\gamma:=\Phi_{|D|}\circ \pi$ is a morphism;

\item the sum of $\pi_{*}^{-1}(D)$ and the exceptional divisors of $\pi$ has simple normal crossing support. 
\end{enumerate}
Let $W \stackrel{f}{\longrightarrow} \Gamma \stackrel{s}{\longrightarrow} Z$ be the Stein factorization of $\gamma$ with $Z:=\gamma(W) \subset\mathbb{P}^{h^0(X, D)-1}$. We have the following commutative diagram:
$$\xymatrix@=4.5em{
W\ar[d]_\pi \ar[dr]^{\gamma} \ar[r]^f& \Gamma\ar[d]^s\\
X \ar@{-->}[r]^{\Phi_{|D|}} & Z.}
$$

If $\dim(\Gamma)\geq2$, then a general member $S$ of $|M|$ is a smooth projective surface by Bertini's theorem. In this case, $|D|$ is said to be \emph{not composed with a pencil of surfaces} (not composed with a pencil, 
for short). 

If $\dim(\Gamma)=1$, then $\Gamma \cong \mathbb{P}^{1}$ as $h^1(\Gamma, \mathcal{O}_\Gamma)\leq h^1(W, \mathcal{O}_W)=h^1(X, \mathcal{O}_X)=0$. Furthermore, a general fiber $S$ of $f$ is a smooth projective surface by Bertini's theorem. In this case, $|D|$ is said to be \emph{composed with a (rational) pencil of surfaces} (composed with a pencil, for short).

In each case, $S$ is called a \emph{generic irreducible element} of $|M|$. 
We can also define a generic irreducible element of a moving linear system on a surface in the similar way.

\begin{definition}
Keep the same notation as above. Let $D'$ be another $\bQ$-Cartier Weil divisor on $X$ with $h^0(X, D')\geq 2$. We say that $|D|$ and $|D'|$ are \emph{composed with the same pencil}, if both of them are composed with pencils and they define the same fibration structure $W\to\bP^1$. In particular, $|D|$ and $|D'|$ are not composed with the same pencil if one of them is not composed with a pencil.
\end{definition}

\subsection{Reid's Riemann--Roch formula and Chen--Chen's method}\label{R-R}
A basket $B$ is a collection of pairs of coprime integers where a pair is allowed to appear several times, say $$\{(b_{i}, r_{i}) \mid i=1, \dots, s ; b_{i} \text{ is coprime to } r_{i}\}.$$ For simplicity, we will alternatively write a basket as a set of pairs with weights, say for example, 
$$
B=\{2\times(1, 2), (1, 3), (3, 7), (5, 11)\}.
$$

Let $X$ be a terminal weak $\mathbb{Q}$-Fano $3$-fold. According to Reid~\cite{YPG}, there is a basket of (virtual) orbifold points
$$
B_{X}=\{(b_{i}, r_{i}) \mid i=1, \dots, s ; 0<b_{i}\leq\frac{r_{i}}{2} ; b_{i} \text{ is coprime to } r_{i}\}
$$
associated to $X$, where a pair $\left(b_{i}, r_{i}\right)$ corresponds to an orbifold point $Q_{i}$ of type $\frac{1}{r_{i}}\left(1, -1, b_{i}\right)$. Denote by $\mathcal{R}_X$ the collection of $r_i$ (counted with multiplicities) appearing in $B_X$, and $r_{\max}=\max\{r_i\mid r_i\in \mathcal{R}_X\}$. Note that the Cartier index $r_X$ of $K_X$ is just $\lcm\{r_i\mid r_i\in \mathcal{R}_X\}$.

According to Reid~\cite{YPG}, 
for any positive integer $n$, 
\begin{align}
P_{-n} 
&=\frac{1}{12} n(n+1)(2 n+1)\left(-K_{X}^{3}\right)+(2 n+1)-l(n+1) \label{RR for P}
\end{align}
where $l(n+1)=\sum_{i} \sum_{j=1}^{n} \frac{\overline{j b_{i}}\left(r_{i}-\overline{j b_{i}}\right)}{2 r_{i}}$ and the first sum runs over Reid's basket of orbifold points. Here $\overline{jb_i}$ means the smallest non-negative residue of $jb_i \bmod r_i$.

Set $\sigma(B_X)=\sum_{i}b_{i}$ and $\sigma'(B_X)=\sum_{i}\frac{b_{i}^{2}}{r_{i}}$.
From \eqref{RR for P} for $n=1, 2$, 
\begin{align}
-K_X^{3}&=2P_{-1}+\sigma(B_X)-\sigma'(B_X)-6;
\label{volume}\\
\sigma(B_X)&=10-5P_{-1}+P_{-2}.\label{sigma}
\end{align}
Denote 
$$
\gamma(B_X):=\sum_i\frac{1}{r_i}-\sum_ir_i+24.
$$
By \cite{KMMT} and \cite[10.3]{YPG}, 
\begin{align}
 \gamma(B_X) \geq 0. \label{gamma>0}
\end{align}

We recall Chen--Chen's method on basket packing from \cite{CC08}.
Given a basket 
$$B=\{(b_{i}, r_{i}) \mid i=1, \dots, s ; 0<b_{i}\leq\frac{r_i}{2} ; b_{i} \text{ is coprime to } r_{i}\}$$
and assume that $b_1r_2-b_2r_1=1$, then the new basket $$B'=\{(b_1+b_2, r_1+r_2), (b_3, r_3), \dots, (b_s, r_s)\}$$ is called a \emph{prime packing} of $B$.
We say that a basket $B'$ is \emph{dominated by} $B$, denoted by $B\succeq B'$, if $B'$ can be achieved from $B$ by a sequence of prime packings (including the case $B=B'$).

By \cite[\S 2.5]{CC08}, there is a unique basket $B_{X}^{(0)}$, called {\it the initial basket of $B_X$}, of the form $B_{X}^{(0)}=\{n_{1, r}^{0}\times(1, r)\mid r\geq2\}$ such that $B_{X}^{(0)}\succeq B_X$.
By \cite[\S 2.7]{CC08}, we have
\begin{align}
n_{1, 2}^{0}{}&=5-6P_{-1}+4P_{-2}-P_{-3} ; \label{n12}\\
n_{1, 3}^{0}{}&=4-2P_{-1}-2P_{-2}+3P_{-3}-P_{-4} ; \label{n13} \\
n_{1, 4}^{0}{}&=1+3P_{-1}-P_{-2}-2P_{-3}+P_{-4}-\sigma_{5}, \label{n14} \end{align}
where $\sigma_{5}=\sum_{r\geq5} n_{1, r}^{0}$. We refer to \cite{CC08} for more details.

\subsection{Auxiliary results}

We list here some useful results on terminal weak $\bQ$-Fano $3$-folds.

\begin{prop}\label{known}
Let $X$ be a terminal weak $\bQ$-Fano $3$-fold. Then
\begin{enumerate}[(1)]
 \item $r_X=840$ or $r_X\leq660$ (\cite[Proposition~2.4]{CJ16});
 
 \item $P_{-8}\geq2$ (\cite[Theorem~1.1]{CC08}); moreover, if $P_{-1}=0$ and $P_{-2}>0$, then $P_{-6}\geq2$ (\cite[Case 1 of Proof of Proposition~3.10]{CC08});
 
 \item $-K_X^3\geq\frac{1}{330}$ (\cite[Theorem~1.1]{CC08}); moreover, if $P_{-1}=0$ and $P_{-2}>0$, then $-K_X^3\geq\frac{1}{70}$, and if in addition $P_{-4}\geq2$, then $-K_X^3\geq\frac{1}{30}$ (\cite[(4.1), Lemma~4.2, Case I of Proof of Theorem~4.4]{CC08});
 
 \item if $P_{-1}=0$, then $2\in\mathcal{R}_X$ (\cite[Proof of Theorem~1.8, Page 106]{CJ16}).

\end{enumerate}
\end{prop}

\begin{lem}\label{3/2bi}Suppose that $\{(b_i, r_i)\mid 1\leq i\leq k\}$ is a collection of pairs of integers with $0<2b_i\leq r_i$ for $1\leq i\leq k$. Then
$\sum_{i=1}^k(r_i-\frac{1}{r_i})\geq \frac{3}{2}\sum_{i=1}^kb_i$.
\end{lem}

\begin{proof}
$r_i\geq 2b_i$ implies that $r_i-\frac{1}{r_i}\geq \frac{3}{2}b_i$.
\end{proof}
\section{The criteria for birationality}\label{sec 3}
In this section we recall the birationality criteria of terminal weak $\bQ$-Fano $3$-folds in \cite{CJ16, CJ20}. 
Here we remark that all birationality criteria in this section are from \cite{CJ16, CJ20} except for Theorem~\ref{thm bc 2} and Corollary~\ref{usage of thm bc 2} (which are minor generalizations of \cite[Theorem~5.9]{CJ20}). Also we provide Lemma~\ref{N_0} in order to apply Corollary~\ref{usage of thm bc 2} efficiently. In fact, in \cite{CJ20}, \cite[Theorem~5.9]{CJ20} is only used for very special cases, but in this paper, thanks to Lemma~\ref{N_0}, we make use of Corollary~\ref{usage of thm bc 2} in many cases.

\subsection{General settings} 

We recall numerical invariants needed in the birationality criteria, namely, $\nu_0$, $m_0$, $a(m_0)$, $m_1$, $\mu'_0$, $N_0$.

\begin{notation}\label{set up}

Let $X$ be a terminal weak $\bQ$-Fano $3$-fold. 

Let $\nu_{0}$ be a positive integer such that $P_{-\nu_{0}}>0$. 

Take a positive integer $m_0$ such that $P_{-m_0}\geq2$. Set
$$
a(m_{0})= \begin{cases}6, & \text{if } m_{0} \geq 2; \\ 1, & \text{if } m_{0}=1.\end{cases}
$$

Take $m_1\geq m_0$ to be an integer with $P_{-m_1}\geq 2$ such that $|-m_0K_X|$ and $|-m_1K_X|$ are not composed with the same pencil. 

Set $D:=-m_0K_X$ and keep the same notation as in Subsection~\ref{r-m}. 
Denote $S$ to be a generic irreducible element of $|M_{-m_0}|=\Mov|\lfloor\pi^*(-m_0K_X)\rfloor|$. 
Choose a positive rational number $\mu_{0}'$ such that
$$
\mu_{0}' \pi^{*}(-K_{X})-S \sim_{\bQ} \text{effective } \bQ\text{-divisor.}
$$
Set $N_{0}=r_{X}(\pi^{*}(-K_{X})^{2} \cdot S)$. 

\end{notation}

\begin{remark}[{\cite[Remark~5.8]{CJ20}}] \label{rem mu0}
Here we explain how to choose $\mu'_0$. In general, by assumption, we can always take $\mu_{0}'=m_0$. On the other hand, if $|-m_{0} K_{X}|$ and $|-k K_{X}|$ are composed with the same pencil for some positive integer $k$, and $\frac{k}{P_{-k}-1}<m_0$, then we can take $\mu_{0}'=\frac{k}{P_{-k}-1}$ as $$k \pi^{*}(-K_{X})\sim_{\bQ} (P_{-k}-1)S+\text{effective } \bQ\text{-divisor}.$$ 
\end{remark}

\begin{lem}\label{N_0}
In Notation~\ref{set up}, $N_0\geq\lceil \frac{r_X}{m_1\nu_0 r_{\max}}\rceil$.
\end{lem}
\begin{proof}
We may modify $\pi$ such that $|M_{-m_1}|=\Mov|\lfloor\pi^*(-m_1K_X)\rfloor|$ is base point free. 
Pick a generic irreducible element $C$ of the base point free linear system $|M_{-m_1}|_S|$. Since $\pi^*(-m_1K_X)\geq M_{-m_1}$, 
$\pi^*(-m_1K_X)|_S\geq C.$
Set
$$
\zeta:=(\pi^*(-K_X)\cdot C)=(\pi^*(-K_X)|_S\cdot C)_S.
$$
By \cite[Proposition~5.7(v)]{CJ16}, $\zeta\geq\frac{1}{\nu_0 r_{\max}}$.
Since $\pi^*(-K_X)|_S$ is nef, $$\pi^{*}(-K_{X})^{2} \cdot S\geq\pi^*(-K_X)|_S\cdot\frac{1}{m_1}C\geq \frac{1}{m_1\nu_0 r_{\max}}.$$ Hence $N_0\geq\lceil \frac{r_X}{m_1\nu_0 r_{\max}}\rceil$ as $N_0$ is an integer by \cite[Lemma~4.1]{CJ16}.
\end{proof}

\subsection{Birationality criteria}
We recall the birationality criteria of terminal weak $\mathbb{Q}$-Fano $3$-folds.
\begin{thm}[{\cite[Theorem~5.11]{CJ16}}]\label{thm bc}

Keep the setting in Notation~\ref{set up}. Then the $m$-th anti-canonical map $\varphi_{-m}$ is birational if one of the following conditions holds:
\begin{enumerate}[(1)]
 \item $m \geq \max \{m_{0}+m_{1}+a(m_{0}), \lfloor3\mu'_{0}\rfloor+3 m_{1}\};$

\item $m \geq \max \{m_{0}+m_{1}+a(m_{0}), \lfloor\frac{5}{3}\mu'_{0}+\frac{5}{3}m_{1}\rfloor, \lfloor\mu'_{0}\rfloor+m_{1}+2 r_{\max }\};$

\item $m \geq \max \{m_{0}+m_{1}+a(m_{0}), \lfloor\mu'_{0}\rfloor+m_{1}+2 \nu_0r_{\max }\}.$
\end{enumerate}
\end{thm}

As another criterion, we have the following modification of \cite[Theorem~5.9]{CJ20}.

\begin{thm}\label{thm bc 2}
Keep the setting in Notation~\ref{set up}. Fix a real number $\beta\geq8$. Then the $m$-th anti-canonical map $\varphi_{-m}$ is birational if
$$
m \geq \max\left\{m_{0}+a(m_{0}), \left\lceil\mu_{0}'+\frac{4 \nu_{0} r_{\max }}{1+\sqrt{1-\frac{8}{\beta}}}\right\rceil-1, \lfloor\mu_{0}'+\sqrt{{\beta r_{X}}/{N_0} }\rfloor\right\}.
$$
\end{thm}

\begin{proof}
The proof is the same as \cite[Theorem~5.9]{CJ20} by replacing \cite[Lemma~5.10]{CJ20} with Lemma~\ref{bir for surfaces}.
\end{proof}

\begin{lem}[{\cite[Theorem~2.8]{CCCJ}}]\label{bir for surfaces}
Let $S$ be a smooth projective surface and $L$ be a nef and big $\bQ$-divisor on $S$ satisfying the following conditions:
\begin{enumerate}[(1)]
 \item $L^{2}>\beta$, for some real number $\beta\geq8$, and

\item $(L \cdot C_{P}) \geq\frac{4}{1+\sqrt{1-\frac{8}{\beta}}}$ for all irreducible curves $C_{P}$ passing through any very general point $P \in S$.
\end{enumerate}
Then the linear system $|K_{S}+\lceil L\rceil|$ separates two distinct points in very general positions. Consequently, $|K_{S}+\lceil L\rceil|$ gives a birational map.
\end{lem}

We will use the following version of Theorem~\ref{thm bc 2}.

\begin{cor}\label{usage of thm bc 2}
Keep the setting in Notation~\ref{set up}. Then the $m$-th anti-canonical map $\varphi_{-m}$ is birational if one of the following conditions holds:
\begin{enumerate}[(1)]
\item $m \geq \max\{m_{0}+a(m_{0}), \lceil\mu_{0}'\rceil+4 \nu_{0} r_{\max }-1, \lfloor\mu_{0}'+\sqrt{{8r_{X}}/{N_0} }\rfloor\};$

 \item $\nu_0 r_{\max}\geq\sqrt{\frac{r_{X}}{2N_0}}$ and $$
m \geq \max\left\{m_{0}+a(m_{0}), \left\lfloor\mu_{0}'+2\nu_0 r_{\max}+\frac{r_X}{N_0 \nu_0 r_{\max}}\right\rfloor\right\}.
$$
\end{enumerate}

\end{cor}
\begin{proof}
(1) follows directly from \cite[Theorem~5.9]{CJ20} or Theorem~\ref{thm bc 2} with $\beta=8$.
For (2), take 
$$\beta=\frac{N_0}{r_{X}}\left(2\nu_0 r_{\max}+\frac{r_X}{N_0 \nu_0 r_{\max}}\right)^2\geq 8$$
in Theorem~\ref{thm bc 2}.
Then 
\begin{align*}
 \frac{4 \nu_{0} r_{\max }}{1+\sqrt{1-\frac{8}{\beta}}} 
 ={}&\frac{4 \nu_{0} r_{\max }\sqrt{\beta}}{\sqrt{\beta}+\sqrt{\beta-8}}\\
={}&\frac{4 \nu_{0} r_{\max }\sqrt{\beta}}{\sqrt{\frac{N_0}{r_{X}}}\left(2\nu_0 r_{\max}+\frac{r_X}{N_0 \nu_0 r_{\max}}+\left|2\nu_0 r_{\max}-\frac{r_X}{N_0 \nu_0 r_{\max}}\right|\right)}\\
={}&\sqrt{{\beta r_{X}}/{N_0}}.
\end{align*}
So the conclusion follows from Theorem~\ref{thm bc 2}.
\end{proof}

Finally we explain the strategy to apply the birationality criteria to assert the birationality.
It is clear that in order to apply Theorem~\ref{thm bc} and Corollary~\ref{usage of thm bc 2}, we need to control the values of (some of) $\nu_0$, $m_0$, $m_1$, $\mu'_0$, $N_0$ and $r_X, r_{\max}$. To be more precise, we need to give upper bounds of $\nu_0$, $m_0$, $m_1$, $\mu'_0$, $r_X$, $r_{\max}$ and lower bounds of $N_0$. 
Here $m_0$ and $\nu_0$ can be controlled by Proposition~\ref{known} (in particular we can always take $m_0=8$), $\mu'_0$ can be controlled by Remark~\ref{rem mu0}, $N_0$ can be controlled by Lemma~\ref{N_0} (in most cases, we use the trivial lower bound $N_0\geq 1$), and $r_X$, $r_{\max}$ can be controlled by \eqref{gamma>0}. So the most important and difficult part is to bound $m_1$. We will deal with this issue in the next section.

\section{A new criterion for $|-mK|$ not composed with a pencil}\label{sec 4}

In this section, we give a new criterion on when $|-mK_X|$ is not composed with a pencil for a terminal weak $\bQ$-Fano $3$-fold $X$. Such a criterion is essential in order to apply criteria for birationality in Section~\ref{sec 3} (see also \cite{Chen11, CJ16, CJ20}).
In \cite{CJ16}, the following proposition is used to determine when $|-mK|$ is not composed with a pencil.
\begin{prop}[{\cite[Corollary~4.2]{CJ16}}]\label{np1}
Let $X$ be a terminal weak $\bQ$-Fano $3$-fold. If
$$
P_{-m}>r_X(-K_X^3) m+1
$$
for some positive integer $m$, then $|-m K_X|$ is not composed with a pencil.
\end{prop}
However, Proposition~\ref{np1} is too weak for application, especially when $r_X(-K_X^3)$ is large (see Example~\ref{43}). In \cite{CJ20}, there is a modification of this inequality (cf. \cite[Lemma~4.2, Proposition~5.2]{CJ20}) but one has to replace $X$ with a birational model. In this paper, by technique developed recently in \cite{JZ21}, we give a new criterion.
\begin{thm}\label{npthm}
Let $X$ be a terminal weak $\bQ$-Fano $3$-fold. If 
$$
P_{-m}>12m+1
$$
for some positive integer $m$, then $|-mK_X|$ is not composed with a pencil.
\end{thm}

\subsection{A structure theorem of terminal weak $\bQ$-Fano $3$-folds}
We recall the following structure theorem of terminal weak $\bQ$-Fano $3$-folds from \cite{JZ21}. It plays the role of Fano--Mori triples as in \cite{CJ20}.
Unlike \cite{CJ20}, we do not need to replace by a birational model (cf. \cite[Proposition~3.9]{CJ20}). 

\begin{prop}[{\cite[Proposition~4.1]{JZ21}}]\label{prop MFS}
Let $X$ be a terminal weak $\bQ$-Fano $3$-fold. 
Then there exists a normal projective $3$-fold $Y$ birational to $X$ satisfying the following properties:
\begin{enumerate}[(1)]
 \item $Y$ is $\mathbb{Q}$-factorial terminal; 
 \item $-K_Y$ is big;
 \item for any sufficiently large and divisible positive integer $n$, $|-nK_Y|$ is movable;
 \item for a general member $M\in |-nK_Y|$, $M$ is irreducible and $(Y, \frac{1}{n}M)$ is canonical;
 
 \item there exists a projective morphism $g: Y\to S$ with connected fibers where $F$ is a general fiber of $g$, such that one of the following conditions holds:
 \begin{enumerate}
 \item $S$ is a point and $Y$ is a $\bQ$-Fano $3$-fold with $\rho(Y)=1$;
 \item $S= \mathbb{P}^1$ and $F$ is a smooth weak del Pezzo surface;
 \item $S$ is a del Pezzo surface with at worst Du Val singularities and $\rho(S)=1$, and $F\simeq \mathbb{P}^1$.
 \end{enumerate}
\end{enumerate}
\end{prop}
Here we remark that in the proof of 
\cite[Proposition~4.1]{JZ21}, $Y$ is obtained by running a $K$-MMP on a $\bQ$-factorialization of $X$, so the induced map $X\dashrightarrow Y$ is a contraction, that is, it does not extract any divisor.

\subsection{Bounding coefficients of anti-canonical divisors}

In this subsection, we discuss coefficients of certain divisors in the $\mathbb{Q}$-linear system of the anti-canonical divisor in several cases.

\begin{lem}\label{lct}
Let $S$ be a smooth weak del Pezzo surface and $C$ a non-zero effective integral divisor on $S$ which is movable. If $-K_S\sim_{\bQ}aC+B$ for some positive rational number $a$ and some effective $\bQ$-divisor $B$, then $a\leq4$.
\end{lem}

\begin{proof}
By classical surface theory, it is well-known that there is a birational map from $S$ to $\mathbb{P}^2$ or the Hirzebruch surface $\mathbb{F}_0$ or $\mathbb{F}_2$. So by taking pushforward, we may replace $S$ by $\mathbb{P}^2$ or $\mathbb{F}_0$ or $\mathbb{F}_2$. Here $C$ is not contracted by the pushforward as it is movable.

If $S=\mathbb{P}^2$, then intersecting with a general line $L$, we get $a\leq a(C\cdot L)\leq(-K_S\cdot L)=3$. 

If $S=\mathbb{F}_0$, then we may find a ruling structure $\phi: \mathbb{F}_0\to \mathbb{P}^1$ such that $C$ is not vertical. Then we get $a\leq 2$ by intersecting with a fiber of $\phi$.

If $S=\mathbb{F}_2$, then we consider the natural ruling structure $\phi: \mathbb{F}_2\to \mathbb{P}^1$. If $C$ is not vertical, then intersecting with a fiber of $\phi$, we get $a\leq 2$. If $C$ is vertical, then intersecting with $-K_S$ we get
$a(-K_S\cdot C)\leq K_S^2$ which implies that $a\leq 4$.

\end{proof}

\begin{lem}\label{lem picard 1}
Let $Y$ be a $\mathbb{Q}$-factorial terminal $\bQ$-Fano $3$-fold with $\rho(Y)=1$ and $D$ an
integral divisor with $h^0(D)\geq 2$. If $-K_Y\sim_{\bQ}aD+B$ for some positive rational number $a$ and some effective $\bQ$-divisor $B$, then $a\leq 7$. Moreover, the equality holds if and only if $Y\simeq \mathbb{P}(1, 1, 2, 3)$, $B=0$, and $\mathcal{O}_Y(D)\simeq \mathcal{O}_{Y}(1)$.
\end{lem}

\begin{proof}
Suppose that $a\geq 7$. 
As $\rho(Y)=1$, we have $-K_Y\sim_{\bQ}tD$ for some rational number $t\geq a \geq 7$. Recall that the {\it $\mathbb{Q}$-Fano index} of $Y$ is defined by 
 $$q\bQ(Y)=\max \{q \mid-K_{Y} \sim_{\bQ} q A, A \text{ is a Weil divisor}\}.$$
 By \cite[Corollary~3.4(ii)]{Prok10}, $t=q\bQ(Y)\geq 7$. As $h^0(D)\geq 2$, there are two different effective divisors $D_1, D_2\in |D|$ such that 
 $-K_Y\sim_{\bQ}tD_1\sim_{\bQ}tD_2$, which implies that $Y\simeq \mathbb{P}(1, 1, 2, 3)$ by 
 \cite[Theorem~1.4(vi)]{Prok10}. But in this case $t=q\bQ(Y)=7$. Hence $a=7$, $B=0$, and $\mathcal{O}_Y(D)\simeq \mathcal{O}_{Y}(1)$.
\end{proof}

\begin{lem}\label{non-klt}
Keep the setting in Proposition~\ref{prop MFS}, suppose that $S=\mathbb{P}^1$. If $-K_Y\sim_{\bQ}\omega F+E$ for some positive rational number $\omega$ and some effective $\bQ$-divisor $E$, then $\omega\leq12.$
\end{lem}
This lemma is from the proof of \cite[Proposition~4.2]{JZ21}. For the reader's convenience, we recall the proof here.
\begin{proof}
We may assume that $\omega>2$. By Proposition~\ref{prop MFS}(3)(4), for a sufficiently large and divisible integer $n$, $|-nK_Y|$ is movable, and there exists an effective $\bQ$-divisor $M\sim -nK_Y$ such that $(Y, \frac{1}{n}M)$ is canonical. Since $-K_Y$ is big, we can write $-K_Y\sim_{\bQ}A+N$, where $A$ is an ample $\bQ$-divisor and $N$ is an effective $\bQ$-divisor. Set $B_{\epsilon}=\frac{1-\epsilon}{n}M+\epsilon N$ for a rational number $0<\epsilon<1$.
Take two general fibers $F_1, F_2$ of $g$. Denote
$$
\Delta=(1-\frac{2}{\omega})B_{\epsilon}+\frac{2}{\omega}E+F_1+F_2.
$$
Then
$$
-(K_Y+\Delta)\sim_{\bQ}-(1-\frac{2}{\omega})(K_Y+B_{\epsilon})\sim_{\bQ}(1-\frac{2}{\omega})\epsilon A
$$
is ample as $\omega>2$. Hence by the connectedness lemma (\cite[Lemma~2.6]{JZ21}), $\Nklt(Y, \Delta)$ is connected. By construction, $F_1\cup F_2\subset\Nklt(Y, \Delta)$, then $\Nklt(Y, \Delta)$ dominates $\bP^1$. By the inversion of adjunction (\cite[Lemma~5.50]{KM}), $(F, (1-\frac{2}{\omega})B_{\epsilon}|_F+\frac{2}{\omega}E|_F)$ is not klt for a general fiber $F$ of $g$. As being klt is an open condition on the coefficients, by the arbitrariness of $\epsilon$, 
it follows that $(F, (1-\frac{2}{\omega})\frac{1}{n}M|_F+\frac{2}{\omega}E|_F)$ is not klt for a very general fiber $F$ of $g$.

On the other hand, as $(Y, \frac{1}{n}M)$ is canonical, $(F, \frac{1}{n}M|_F)$ is canonical by Bertini's theorem (see \cite[Lemma~5.17]{KM}).
Since $M$ is a general member of a movable linear system by assumption, $M|_F$ is a general member of a movable linear system on $F$. So each irreducible component of $M|_F$ is nef. Also we can take $M$ such that $M|_F$ and $E|_F$ have no common irreducible component. By construction, $\frac{1}{n}M|_F\sim_\bQ E|_F\sim_\bQ -K_F$. 
So we can apply \cite[Theorem~3.3]{JZ21} to $F, \frac{1}{n}M|_F, E|_F$, which implies that $
\frac{2}{\omega}\geq \frac{1}{6}.$ Hence $\omega\leq12$.
\end{proof}
 
\begin{lem}\label{non-klt2}
Keep the setting in Proposition~\ref{prop MFS}, suppose that $S$ is a del Pezzo surface. Suppose that $D$ is a non-zero effective integral divisor on $Y$ which is movable. If $-K_Y\sim_{\bQ}\omega D+E$ for some positive rational number $\omega$ and some effective $\bQ$-divisor $E$, then $\omega\leq12.$
\end{lem}

\begin{proof}
As $S$ is a del Pezzo surface with at worst Du Val singularities and $\rho(S)=1$, there are 3 cases (see \cite{MZ88}, \cite[Remark~3.4(ii)]{Prok07}):
\begin{enumerate}[(1)]
 \item $K_{S}^2=9$ and $S\simeq \mathbb{P}^2$;
 \item $K_{S}^2=8$ and $S\simeq \mathbb{P}(1, 1, 2)$;
 \item $1\leq K_{S}^2\leq 6$.
\end{enumerate}
Consider the linear system $\mathcal{H}$ on $S$ defined by
$$\mathcal{H}:=
\begin{cases}
|\mathcal{O}_{\mathbb{P}^2}(1)| & \text{if } S\simeq \mathbb{P}^2;\\
|\mathcal{O}_{\mathbb{P}(1, 1, 2)}(2)| & \text{if } S\simeq \mathbb{P}(1, 1, 2);\\
|-K_{S}|& \text{if } 2\leq K_{S}^2\leq 6;\\
|-2K_{S}|& \text{if } K_{S}^2=1.
\end{cases}
$$
Then $\mathcal{H}$ is base point free and defines a generically finite map (cf. \cite[Theorem~8.3.2]{Dol12}).
By Bertini's theorem, we can take a general element $H\in \mathcal{H}$ such that $H$ and $G=g^{-1}(H)=g^*H$ are smooth.
Note that for a general fiber $C$ of $g|_G$, $C\simeq \mathbb{P}^1$, $(-K_G\cdot C)=2$, and $G|_G\sim (H^2)\cdot C$.

Note that $g|_G$ is factored through by a ruled surface over $H$, so $K_G^2\leq 8-8g(H)$. Then
\begin{align*}
(-K_Y|_G)^2={}&(-K_{G}+G|_G)^2\\={}&K_{G}^2+4H^2\\\leq{}& 8-8g(H)+4H^2\\={}&-4(K_{S}\cdot H)\leq24.
\end{align*}
By construction, as $|G|$ defines a morphism from $Y$ to a surface and $D$ is movable, $D|_G$ is an effective non-zero integral divisor for a general $G$. So we may write
\begin{align}\label{relation on G}
-K_Y|_G\sim_{\bQ}\omega D|_G +E|_G.
\end{align}
Take a general fiber $C$ of $g|_G$. If $(D|_G \cdot C)\neq0$, then by \eqref{relation on G} intersecting with $C$, $\omega\leq2$. 
If $(D|_G \cdot C)=0$, then $D|_G $ is vertical over $H$ and thus $D|_G$ is numerically equivalent to a multiple of $C$. By Proposition~\ref{prop MFS}(3), $-K_Y|_G$ is nef. Then by \eqref{relation on G} intersecting with $-K_Y|_G$, 
$$
24\geq (-K_Y|_G)^2\geq\omega (-K_Y|_G\cdot D|_G)\geq \omega (-K_Y|_G\cdot C)=2\omega, 
$$
which implies that $\omega\leq12$. 
\end{proof}

\subsection{A new geometric inequality}
Now we are prepared to prove Theorem~\ref{npthm}.
\begin{proof}[Proof of Theorem~\ref{npthm}]
It suffices to show that, if $|-mK_X|$ is composed with a pencil, then 
$
P_{-m}\leq12m+1.
$
 
Take $g: Y\to S$ to be the morphism in Proposition~\ref{prop MFS}.
Take a common resolution $\pi:W\to X$, $q:W\to Y$. We may modify $\pi$ such that $f: W\to \bP^1$ is the fibration induced by $|-mK_X|$ as in Subsection~\ref{r-m}. See the following diagram:
$$\xymatrix{
 & W\ar[dl]_\pi \ar[dr]^q \ar[r]^f& \bP^1 \\
X \ar@{-->}[rr] & & Y \ar[r]^g & S.}
$$
Denote by $F_W$ a general fiber of $f$. Then 
\begin{align}\label{3.1}
\pi^*(-mK_X)\sim (P_{-m}-1)F_W+E, 
\end{align}
where $E$ is an effective $\bQ$-divisor on $W$. Set $\omega=\frac{P_{-m}-1}{m}$. Pushing forward \eqref{3.1} to $Y$, we have 
\begin{align}\label{relation on Y}
-K_Y\sim_{\bQ}\omega q_{*}F_W+E_Y, 
\end{align}
where $E_Y$ is an effective $\bQ$-divisor on $Y$. Note that $q_{*}F_W$ is a general member of a movable linear system.


\medskip

{\bf Case 1.} $S$ is a point.

In this case, $\omega \leq 7$ by \eqref{relation on Y} and Lemma~\ref{lem picard 1}.
 
\medskip

{\bf Case 2.} $S=\bP^1$.

 If $S=\bP^1$ and $q_{*}F_W|_F=0$, then $q_{*}F_W \sim F$ and $-K_Y\sim_{\bQ}\omega F+E_Y$. By Lemma~\ref{non-klt}, $\omega\leq12$.

 If $S=\bP^1$ and $q_{*}F_W|_F\neq0$, then $q_{*}F_W|_F$ is a movable effective non-zero integral divisor on $F$. Restricting \eqref{relation on Y} on $F$, we have 
 $
-K_F\sim_{\bQ}\omega (q_{*}F_W|_F)+E_Y|_F. 
$
 By Lemma~\ref{lct}, $\omega\leq4$.

\medskip

{\bf Case 3.} $S$ is a del Pezzo surface.

In this case, $\omega \leq 12$ by \eqref{relation on Y} and Lemma~\ref{non-klt2}.

\medskip

Combining all above cases, we proved that $\frac{P_{-m}-1}{m}=\omega\leq 12$ as long as $|-mK_X|$ is composed with a pencil.
\end{proof}

Applying Proposition~\ref{np1} and Theorem~\ref{npthm}, we have the following criteria for $|-mK|$ not composed with a pencil (cf. \cite[Proposition~5.4]{CJ20}). 
\begin{prop}\label{np2}
Let $X$ be a terminal weak $\bQ$-Fano $3$-fold. Let $t$ be a positive real number and $m$ a positive integer. If $m\geq t, m\geq\frac{r_{\max}t}{3}$, and one of the following conditions holds:
\begin{enumerate}[(1)]
 \item $m>-\frac{3}{4}+\sqrt{\frac{12}{t\cdot(-K_{X}^3)}+6r_X+\frac{1}{16}}$;

 \item $m>-\frac{3}{4}+\sqrt{\frac{12}{t\cdot(-K_{X}^3)}+\frac{72}{-K_{X}^3}+\frac{1}{16}}$, 
\end{enumerate}
then $|-m K_{X}|$ is not composed with a pencil.
\end{prop}

\begin{proof}
By \cite[Proposition~5.3]{CJ20}, 
$$
P_{-m} \geq \frac{1}{12} m(m+1)(2m+1)(-K_{X}^3)+1-\frac{2m}{t}.
$$
The assumption implies that either
$P_{-m}>r_X(-K_X^3) m+1$ or $P_{-m}>12m+1$.
Hence $|-m K_{X}|$ is not composed with a pencil by Proposition~\ref{np1} and Theorem~\ref{npthm}.
\end{proof}

By the same method we have the following corollary.
\begin{cor}\label{np2 cor}
Let $X$ be a terminal weak $\bQ$-Fano $3$-fold. Let $t$ be a positive real number and $m$ a positive integer. If $m\geq t, m\geq\frac{r_{\max}t}{3}$, and $m>-\frac{3}{4}+\sqrt{\frac{12}{t\cdot(-K_{X}^3)}+\frac{6}{l\cdot (-K_{X}^3)}+\frac{1}{16}}$ for some positive real number $l$,
then $P_{-m}-1>\frac{m}{l}$. 
\end{cor}

We illustrate in the following example on how efficient Proposition~\ref{np2} is comparing to \cite[Corollary~4.2]{CJ16}.

\begin{example}\label{43}
 Suppose that $X$ is a terminal weak $\mathbb{Q}$-Fano $3$-fold with $P_{-1}=0$, $B_X=\{2\times(1, 2), (2, 5), (3, 7), (4, 9)\}$, and $-K_X^3=\frac{43}{315}$.
Then \cite[Corollary~4.2]{CJ16} implies that $|-mK_X|$ is not composed with a pencil for all $m\geq 61$ (see the last paragraph of \cite[Page 106]{CJ16}). On the other hand, by Proposition~\ref{np2}, 
$|-mK_X|$ is not composed with a pencil for all $m\geq 23$ (see Case~4 of Proof of Theorem~\ref{case3}), which significantly improves the previous result. Also it can be computed directly by \eqref{RR for P} to get $P_{-22}=260<12\times 22+1$, which tells that the estimates in Proposition~\ref{np2} is efficient enough comparing to directly using the Riemann--Roch formula.
\end{example}

\subsection{A remark on {\cite[Corollary~4.2]{CJ16}}}

In this subsection, we discuss the equality case of {\cite[Corollary~4.2]{CJ16}} for terminal $\bQ$-Fano $3$-folds.
\begin{prop}\label{np1 equality}
Let $X$ be a terminal $\bQ$-Fano $3$-fold and let $m$ be a positive integer. If
$$
P_{-m}=r_X(-K_X^3) m+1
$$
and $|-m K_X|$ is composed with a pencil, then
\begin{enumerate}[(1)]
 \item $r_X(-K_X^3)=1$;
 \item if moreover the Weil divisor class group of $X$ has no $m$-torsion element, then $h^0(X, -kK_X)$
 $=k+1$ for all $1\leq k\leq m$.
\end{enumerate}
\end{prop}
\begin{proof}
We recall the proof of \cite[Corollary~4.2]{CJ16}. 
As $|-m K_X|$ is composed with a pencil, take $D=-mK_X$ and keep the notation in Subsection~\ref{r-m}, 
we have 
\begin{align}
 \pi^*(-mK_X)\sim (P_{-m}-1)S+F\label{-K=S}
\end{align}
where $S$ is a generic irreducible element of $\Mov|\lfloor \pi^*(-mK_X) \rfloor|$ and $F$ is an effective $\mathbb{Q}$-divisor.
Then 
$$
m(-K_X^3)\geq (P_{-m}-1)(\pi^*(-K_X)^2\cdot S)\geq \frac{1}{r_X}(P_{-m}-1)
$$
by \cite[Lemma~4.1]{CJ16}.

Now by assumption, the equality holds. So 
$(\pi^*(-K_X)^2\cdot F)=0$. This implies that $F$ is $\pi$-exceptional as $-K_X$ is ample. So \eqref{-K=S} implies that 
$$
-mK_X\sim (P_{-m}-1)\pi_*S.
$$
Then 
\begin{align}
 -K_X\sim_\bQ r_X(-K_X^3)\pi_*S.\label{-K=nS} 
\end{align}
By \cite[Lemma~2.3]{Jiang16},
$(\pi_*S)^3\geq \frac{1}{r_X}$.
Then
\eqref{-K=nS} implies that
$$
-K_X^3\geq \frac{(r_X(-K_X^3))^3}{r_X},
$$
which implies that $r_X(-K_X^3)=1$ as it is a positive integer. Under the assumption that the Weil divisor class group of $X$ has no $m$-torsion element, \eqref{-K=nS} implies that $-K_X\sim \pi_*S$. So the conclusion follows as $-kK_X\sim k\pi_*S$ is composed with a pencil for any $1\leq k\leq m$ (see \cite[Page~63, Case ($f_p$)]{CJ16}).
\end{proof}

The following example shows that Proposition~\ref{np1 equality} is non-empty. 
\begin{example}[{\cite[List~16.6, No.~88]{IF00}}]
A general weighted hypersurface $X_{42}\subset\mathbb{P}(1, 1,$ $6, 14, 21)$ is a terminal $\mathbb{Q}$-Fano $3$-fold with $r_X(-K_X^3)=1$ and $B_X=\{(1, 2), (1, 3), (1, 7)\}$. By \cite[Proposition~2.9]{Prok10}, the Weil divisor class group of $X$ is torsion free. Certainly, $P_{-k}=k+1$ and $|-kK_X|$ is composed with a pencil for $1\leq k\leq5$.
\end{example}

\section{Proof of main results}\label{sec 5}
In this section, we apply the birationality criteria (Theorem~\ref{thm bc} and Corollary~\ref{usage of thm bc 2}) and the non-pencil criteria (Proposition~\ref{np2}) to prove the main theorem. The proof will be divided into several cases:
\begin{enumerate}[1.]
 \item $r_X=840$;
 \item $P_{-2}=0$;
 \item $P_{-2}>0$, $P_{-1}=0$, and $r_{\max}\geq 14$;
 \item $P_{-2}>0$, $P_{-1}=0$, and $r_{\max}\leq 13$;
 \item $P_{-1}>0$ and $r_{\max}\geq 14$;
 \item $P_{-1}>0$ and $r_{\max}\leq 13$.
\end{enumerate}
Here recall that $r_X=\lcm\{r_i\mid r_i\in \mathcal{R}_X\}$ is the Cartier index of $K_X$, and $r_{\max}=\max\{r_i\mid r_i\in \mathcal{R}_X\}$ is the maximal local index.
\subsection{The case $r_{X}=840$}
\begin{thm}\label{case0}
Let $X$ be a terminal weak $\bQ$-Fano $3$-fold with $r_{X}=840$. Then $\varphi_{-m}$ is birational for all $m\geq48$.
\end{thm}

\begin{proof}
Keep the setting in Notation~\ref{set up}. 
By \cite[Lemma~6.5]{CJ20} and the first line of its proof, we know that $r_{\max}=8$, $P_{-1}\geq1$, and $-K_X^3\geq\frac{47}{840}$. Take $m_0=8$ and $\nu_0=1$.
By Corollary~\ref{np2 cor} (with $l=1$, $t=4.5$, $-K_X^3\geq\frac{47}{840}$), we have $P_{-12}-1>12.$

If $|-12K_X|$ and $|-8K_X|$ are composed with the same pencil, then take $\mu_0'=\frac{12}{P_{-12}-1}<1$ by Remark~\ref{rem mu0}. By Proposition~\ref{np2}(2) (with $t=13.5$, $-K_X^3\geq\frac{47}{840}$), we can take $m_1=36$. 
By Lemma~\ref{N_0}, 
$N_0\geq \roundup{\frac{840}{8m_1}}= 3.$
Then by Corollary~\ref{usage of thm bc 2}(1), $\varphi_{-m}$ is birational for all $m\geq48$.

If $|-12K_X|$ and $|-8K_X|$ are not composed with the same pencil, then take $m_1=12$ and $\mu_0'=m_0=8$. Then by Theorem~\ref{thm bc}(3), $\varphi_{-m}$ is birational for all $m\geq36$.
\end{proof}

\subsection{The case $P_{-2}=0$}

\begin{thm}\label{case1}
Let $X$ be a terminal weak $\bQ$-Fano $3$-fold. If $P_{-2}=0$, then $\varphi_{-m}$ is birational for all $m\geq51$.
\end{thm}

\begin{proof}
Keep the setting in Notation~\ref{set up}.
In this case, the possible baskets are classified in \cite[Theorem~3.5]{CC08} with 23 cases in total (see Table~\ref{tabA} in Appendix).
Here we refer to the numbering in Table~\ref{tabA}.

For No.~1-5 of Table~\ref{tabA}, $r_{X} \leq 84, -K_{X}^{3} \geq \frac{1}{84}$, $P_{-8}\geq2$, and $r_{\max } \leq 11$. So we can take $m_0=8$. By Corollary~\ref{np2 cor} (with $l=2, t=5.7$), $P_{-21}-1>\frac{21}{2}$. If $|-21K_X|$ and $|-8K_X|$ are composed with the same pencil, then take $\mu_0'=\frac{21}{P_{-21}-1}<2$ by Remark~\ref{rem mu0}. 
By Proposition~\ref{np2}(1) (with $t=6.6$), we can take $m_{1}=25$.
Then by Theorem~\ref{thm bc}(2), $\varphi_{-m}$ is birational for all $m \geq 48$.
If $|-21K_X|$ and $|-8K_X|$ are not composed with the same pencil, then take $m_1=21$ and $\mu_{0}'=m_0=8$. Then by Theorem~\ref{thm bc}(2), $\varphi_{-m}$ is birational for all $m \geq 51$.

For No.~6-13 and No.~16-23 of Table~\ref{tabA}, $r_{X} \leq 78, -K_{X}^{3} \geq \frac{1}{30}$, $P_{-6}\geq2$, and $r_{\max } \leq 14$. So we can take $m_0=6$. By Corollary~\ref{np2 cor} (with $l=1, t=3.6$), $P_{-17}-1>17$. If $|-17K_X|$ and $|-6K_X|$ are composed with the same pencil, then take
$\mu_0'=\frac{17}{P_{-17}-1}<1$ by Remark~\ref{rem mu0}.
By Proposition~\ref{np2}(1) (with $t=4.9$), we can take $m_{1}=23$.
Then by Theorem~\ref{thm bc}(2), $\varphi_{-m}$ is birational for all $m \geq 51$.
If $|-17K_X|$ and $|-6K_X|$ are not composed with the same pencil, then take $m_1=17$ and $\mu_{0}'=m_0=6$. Then by Theorem~\ref{thm bc}(2), $\varphi_{-m}$ is birational for all $m \geq 51$.

For No.~14 of Table~\ref{tabA}, $r_{X}=210, -K_{X}^{3}=\frac{17}{210}$, $P_{-5}=2$, and $r_{\max}=7$. So we can take $m_0=5$. By Corollary~\ref{np2 cor} (with $l=1, t=4.2$), $P_{-10}-1>10$.
If $|-10K_X|$ and $|-5K_X|$ are composed with the same pencil, then take $\mu_0'=\frac{10}{P_{-10}-1}<1$ by Remark~\ref{rem mu0}.
By Proposition~\ref{np2}(2) (with $t=12$), we can take $m_{1}=30$.
Then by Theorem~\ref{thm bc}(2), $\varphi_{-m}$ is birational for all $m \geq 51$.
If $|-10K_X|$ and $|-5K_X|$ are not composed with the same pencil, then take $m_1=10$ and $\mu_{0}'=m_0=5$. Then by Theorem~\ref{thm bc}(2), $\varphi_{-m}$ is birational for all $m \geq 29$.

For No.~15 of Table~\ref{tabA}, $r_{X}=120, -K_{X}^{3}=\frac{3}{40}$, $P_{-5}=2$, and $r_{\max }=8$. So we can take $m_0=5$. By Corollary~\ref{np2 cor} (with $l=1, t=4$), $P_{-11}-1>11$.
If $|-11K_X|$ and $|-5K_X|$ are composed with the same pencil, then take $\mu_0'=\frac{11}{P_{-11}-1}<1$ by Remark~\ref{rem mu0}.
By Proposition~\ref{np2}(1) (with $t=10$), we can take $m_{1}=27$.
Then by Theorem~\ref{thm bc}(2), $\varphi_{-m}$ is birational for all $m \geq 46$.
If $|-11K_X|$ and $|-5K_X|$ are not composed with the same pencil, then take $m_1=11$ and $\mu_{0}'=m_0=5$. Then by Theorem~\ref{thm bc}(2), $\varphi_{-m}$ is birational for all $m \geq 32$.
\end{proof}
 
\subsection{The case $P_{-2}>0$ and $P_{-1}=0$}

\begin{lem}\label{lem:5.3}
Let $X$ be a terminal weak $\bQ$-Fano $3$-fold. If $P_{-1}=0$ and $P_{-2}>0$, then
\begin{align}
 \gamma(B_X)\geq 0, \quad 2\in\mathcal{R}_X, \quad \sigma(B_X) \geq11. \label{condition 014}
\end{align}
\end{lem}

\begin{proof}
By \eqref{sigma}, we have $\sigma(B_X)=10-5P_{-1}+P_{-2}=10+P_{-2}\geq11$.
Other statements follow from \eqref{gamma>0} and Proposition~\ref{known}(4). 
\end{proof}

\begin{thm}\label{case2}
Let $X$ be a terminal weak $\bQ$-Fano $3$-fold. If $P_{-2}>0$, $P_{-1}=0$, and $r_{\max}\geq14$, then $\varphi_{-m}$ is birational for all $m\geq59$.
Moreover, $\varphi_{-58}$ may not be birational only if $B_X=\{(1, 2), 2\times (1, 3), (8, 17)\}$ and $|-24K_X|$ is composed with a pencil.
\end{thm}

\begin{proof}
Keep the setting in Notation~\ref{set up}.
By \cite[Case~II of Proof of Theorem~3.12]{CJ16} (especially the last paragraph of Subsubcase II-3-iii) or the second paragraph of \cite[Case~IV of Proof of Theorem~1.8]{CJ16} (see Table~\ref{tabB} in Appendix), we can see that $r_{\max}\leq 13$ provided $P_{-4}=1$. Hence by assumption, $P_{-4}\geq2$ and we can always take $m_0=4$. 

\medskip

{\bf Case 1}. $ r_{\max}\geq 16$.

It is not hard to search by hands or with the help of a computer program to get all possible $B_X$ satisfying \eqref{condition 014} and $r_{\max}\geq 16$. Here note that $\sigma(B_X) \geq11$ implies that $\sum_i r_i> 2\sigma(B_X)\geq 22.$

If $22\leq r_{\max}\leq 24$, then there is no $B_X$ satisfying \eqref{condition 014}.
If $16\leq r_{\max}\leq 21$, then all possible $B_X$ satisfying \eqref{condition 014} are listed in Table~\ref{tab1}.

Here we explain briefly how to get Table~\ref{tab1}. The algorithm is the follwoing: first we can list all possible $\mathcal{R}_X$ satisfying $2\in \mathcal{R}_X$ and $\gamma(B_X)\geq 0$; then we find all possible $b_i$ for those $\mathcal{R}_X$ such that $\sigma(B_X)\geq 11$.
For example, let us consider the case $r_{\max}=17$. As $2\in \mathcal{R}_X$, $\{2, 17\}\subset \mathcal{R}_X$. So we can list all possible $\mathcal{R}_X$ with $\gamma(B_X)\geq 0$ by enumeration method by considering the second largest $r_i$:
\begin{center}
\begin{tabular}{ l l l }
 \{2, 5, 17\}; &\{2, 2, 4, 17\}; & \{2, 4, 17\};\\
\{2, 3, 3, 17\}; &\{2, 2, 3, 17\}; &\{2, 3, 17\};\\
\{2, 2, 2, 2, 17\}; &\{2, 2, 2, 17\}; &\{2, 2, 17\};\\ 
\{2, 17\}. &&   
\end{tabular}
\end{center}
Then all possible $B_X$ with $\sigma(B_X)\geq 11$ are listed in Table~\ref{tab1}; for instance, there is no  such basket $B_X$ with $\mathcal{R}_X=\{2,4, 17\}$ because in this case $\sigma(B_X)\leq 1+1+8=10.$

{\footnotesize
\begin{longtable}{CLC}
\caption{Baskets satisfying  Lemma~\ref{lem:5.3} with $r_{\max}\geq 16$}\label{tab1}\\
\hline
\text{No.} & B_X & -K^3 \\
\hline
\endfirsthead
\multicolumn{3}{l}{{ {\bf \tablename\ \thetable{}} \textrm{-- continued from previous page}}}
\\
\hline 
\text{No.} & B_X & -K^3 \\
\endhead
\hline
\hline \multicolumn{3}{r}{{\textrm{Continued on next page}}} \\ \hline
\endfoot

\hline \hline
\endlastfoot

1& \{(1, 2), (10, 21)\} & <0\\
2& \{2\times (1, 2), (10, 21)\} & 5/21\\
3& \{2\times (1, 2), (9, 20)\} & <0\\
4& \{2\times (1, 2), (9, 19)\} & <0\\
5& \{ (1, 2), (1, 3), (9, 19)\} & <0\\
6& \{3\times (1, 2), (9, 19)\} & 9/38\\
7& \{3\times (1, 2), (8, 19)\} & 5/38\\
8& \{4\times (1, 2), (7, 18)\} & 5/18\\
9& \{(1, 2), (2, 5), (8, 17)\} & <0\\
10& \{3\times (1, 2), (8, 17)\} & <0\\
11& \{2\times (1, 2), (1, 3), (8, 17)\} & <0\\
12& \{2\times (1, 2), (1, 4), (8, 17)\} & <0\\
13& \{(1, 2), 2\times (1, 3), (8, 17)\} & 7/102\\
14& \{4\times (1, 2), (8, 17)\} & 4/17\\
15& \{4\times (1, 2), (7, 17)\} & 2/17\\
16& \{2\times (1, 2), (2, 5), (7, 16)\} & 11/80\\
17& \{4\times (1, 2), (7, 16)\} & <0\\
18& \{3\times (1, 2), (1, 3), (7, 16)\} & 5/48\\
19& \{5\times (1, 2), (7, 16)\} & 7/16\\
\hline
\end{longtable}
}

For No.~2 of Table~\ref{tab1}, $-K_X^3=\frac{5}{21}$ and in this case $P_{-2}=2$, $r_X=42$, $r_{\max}=21$. We can take $\mu_{0}'=m_0=2$. By Proposition~\ref{np2}(1) (with $t=2.28$), we can take $m_1=16$. Then by Theorem~\ref{thm bc}(1), $\varphi_{-m}$ is birational for all $m \geq 54$.

For other cases with $-K_X^3>0$, we have $-K_X^3\geq\frac{7}{102}$ and $r_{\max}\leq19$.
By Corollary~\ref{np2 cor} (with $l=1$, $t=2$), $P_{-13}-1>13$. If $|-13K_X|$ and $|-4K_X|$ are not composed with the same pencil, then take $m_1=13$ and $\mu_{0}'=m_0=4$. Then by Theorem~\ref{thm bc}(1), $\varphi_{-m}$ is birational for all $m\geq51$. 

So we may assume that $|-13K_X|$ and $|-4K_X|$ are composed with the same pencil. Then in the following we can take $\mu_0'=\frac{13}{P_{-13}-1}<1$ by Remark~\ref{rem mu0} and $m_0=4$. 
 
For No.~6-8 of Table~\ref{tab1}, $-K_X^3\geq\frac{5}{38}$, $r_X\leq 38$, $r_{\max}\leq 19$. By Proposition~\ref{np2}(1) (with $t=2.4$), we can take $m_1=16$. Then by Theorem~\ref{thm bc}(1), $\varphi_{-m}$ is birational for all $m \geq 50$.

For No.~14-16 and No.~18-19 of Table~\ref{tab1}, $-K_X^3\geq\frac{5}{48}$, $r_X\leq 80$, $r_{\max}\leq 17$. By Proposition~\ref{np2}(1) (with $t=3.88$), we can take $m_1=22$. Then by Theorem~\ref{thm bc}(2), $\varphi_{-m}$ is birational for all $m \geq 56$.

For No.~13 of Table~\ref{tab1}, $-K_X^3=\frac{7}{102}$, $r_X=102$, $r_{\max}=17$. By Proposition~\ref{np2}(1) (with $t=3.6$), we can take $m_1=25$. Then by Theorem~\ref{thm bc}(2), $\varphi_{-m}$ is birational for all $m \geq 59$. Moreover, if $|-24K_X|$ is not composed with a pencil, then take $m_1=24$. Then by Theorem~\ref{thm bc}(2), $\varphi_{-m}$ is birational for all $m \geq 58$.

\medskip

{\bf Case 2}. $ 14\leq r_{\max}\leq 15$.

For the remaining cases $14\leq r_{\max}\leq15$, 
we have $-K_X^3\geq\frac{1}{30}$ by Proposition~\ref{known}(3) as $P_{-4}\geq 2$. 
By Corollary~\ref{np2 cor} (with $l=1, t=3.4$), $P_{-17}-1>17$.

If $|-17K_X|$ and $|-4K_X|$ are not composed with the same pencil, then take $m_1=17$ and $\mu_0'=m_0=4$. Then by Theorem~\ref{thm bc}(2), $\varphi_{-m}$ is birational for all $m\geq51$.

If $|-17K_X|$ and $|-4K_X|$ are composed with the same pencil, then take $\mu_0'=\frac{17}{P_{-17}-1}<1$ by Remark~\ref{rem mu0} and $m_0=4$.

If $r_{\max}=15$, then we claim that $r_X\leq 60$. In fact, as $\{15, 2\}\subset \mathcal{R}_X$, by \eqref{gamma>0}, $s\not \in \mathcal{R}_X$ for all $s>7$. If $7\not\in \mathcal{R}_X$, then $r_X$ divides $60$. If $7\in \mathcal{R}_X$, then $\mathcal{R}_X=\{15, 7, 2\}$ by \eqref{gamma>0}, and moreover $B_X=\{(1, 2), (3, 7), (7, 15)\}$ by $\sigma(B_X)\geq 11$. But this basket has $-K_X^3<0$ by \eqref{volume}, which is absurd. Hence $r_X\leq60$. By Proposition~\ref{np2}(1) (with $t=4$, $-K_X^3\geq\frac{1}{30}$), we can take $m_1=21$. 
Then by Theorem~\ref{thm bc}(2), $\varphi_{-m}$ is birational for all $m\geq51$.

If $r_{\max}=14$, then we claim that $B_X=\{6\times (1, 2), (5, 14)\}$. Suppose that $B_X=\{(b_1, r_1), \dots, (b_k, r_k), (b, 14)\}$ with $b\in\{1,3,5\}$. 
Then $\sigma(B_X)\geq11$ implies that $\sum_{i=1}^kb_i\geq 6$. If there exists some $r_i>2$, 
then $\sum_{i=1}^kr_i>2\sum_{i=1}^kb_i\geq 12$, that is, $\sum_{i=1}^kr_i\geq 13$. So \eqref{gamma>0} implies that $\sum_{i=1}^k\frac{1}{r_i}\geq 3-\frac{1}{14}$.
On the other hand, \eqref{gamma>0} implies that $\frac{3}{2}k+14-\frac{1}{14}\leq 24$, which says that $k\leq 6.$ So $\sum_{i=1}^k\frac{1}{r_i}\leq 5\times \frac{1}{2}+\frac{1}{3}<3-\frac{1}{14}$, a contradiction. So all $r_i= 2$ and $k\geq 6$. Then $\gamma(B_X)\geq0$ implies that $k=6$, and $\sigma(B_X)\geq 11$ implies that $b=5$. We conclude that $B_X=\{6\times (1, 2), (5, 14)\}$.
In this case, $-K_X^3=\frac{3}{14}$, $r_X=r_{\max}=14$. By Proposition~\ref{np2}(1) (with $t=2$), we can take $m_1=10$.
Then by Theorem~\ref{thm bc}(1), $\varphi_{-m}$ is birational for all $m\geq32$.

Combining all above cases, we have proved the theorem.
\end{proof}

\begin{lem}[{cf. \cite[Case I of Proof of Theorem~4.4]{CC08}}]\label{1/12}
Let $X$ be a terminal weak $\bQ$-Fano $3$-fold with $P_{-1}=0$ and $P_{-2}>0$. If $P_{-4}\geq2$ and $-K_X^3<\frac{1}{12}$, then $B_X$ is dominated by one of the following initial baskets:
\begin{align*}
 {}&\{8\times(1, 2), 3\times(1, 3)\};\\
 {}&\{9\times(1, 2), (1, 4), (1, 5)\};\\
 {}&\{9\times(1, 2), (1, 4), (1, 6)\};\\
 {}&\{9\times(1, 2), 2\times (1, 5)\}.
\end{align*}
Note that in the latter $3$ cases, all possible packings have $r_{\max }\leq9$.
\end{lem}

\begin{proof}
Following \cite[Case I of Proof of Theorem~4.4]{CC08}, we only need to consider the cases $(P_{-3}, P_{-4})=(1, 2)$ or $(0, 2)$ in \cite[Subcase I-3 of Proof of Theorem~4.4]{CC08}.

If $(P_{-3}, P_{-4})=(1, 2)$, then \cite[Subcase I-3 of Proof of Theorem~4.4]{CC08} shows that $B_X$ is dominated by $\{8\times(1, 2), 3\times(1, 3)\}$. (Actually it shows moreover that $B_X$ is dominated by $\{7\times(1, 2), (2, 5), 2\times(1, 3)\}$.)

If $(P_{-3}, P_{-4})=(0, 2)$, then $P_{-1}=0$ and $P_{-2}=1$. Then by \eqref{n12}--\eqref{n14}, $n^0_{1, 2}=9$, $n^0_{1, 3}=0$, $n^0_{1, 4}+\sigma_5=2$. So $B_X$ is dominated by $\{9\times(1, 2), (1, s_1), (1, s_2)\}$ for some $s_2\geq s_1\geq 4$. The case $(s_1, s_2)=(4, 4)$ is ruled out by 
\cite[Subcase I-3 of Proof of Theorem~4.4]{CC08}. Hence we get the conclusion by \eqref{gamma>0} and \cite[Lemma~3.1]{CC08}.
\end{proof}

\begin{thm}\label{case3}
Let $X$ be a terminal weak $\bQ$-Fano $3$-fold. If $P_{-1}=0, P_{-2}>0$, and $r_{\max}\leq 13$, then $\varphi_{-m}$ is birational for all $m\geq56$.
\end{thm}

\begin{proof}
Keep the setting in Notation~\ref{set up}.
By Proposition~\ref{known}, $P_{-6}\geq2$ and $-K_X^3\geq\frac{1}{70}$.
We always take $\nu_0=2$.

If $P_{-4}=1$, then $P_{-2}=1$. Following the second paragraph of \cite[Case IV of Proof of Theorem~1.8]{CJ16} (see Table~\ref{tabB} in Appendix), we have $r_X\leq130$. By Corollary~\ref{np2 cor} (with $l=1, t=5.5$, $-K_X^3\geq\frac{1}{70}$), $P_{-24}-1>24$. If $|-24K_X|$ and $|-6K_X|$ are not composed with the same pencil, then take $m_1=24$ and $\mu_0'=m_0=6$. Then by Theorem~\ref{thm bc}(2), $\varphi_{-m}$ is birational for all $m\geq56$. If $|-24K_X|$ and $|-6K_X|$ are composed with the same pencil, then take $m_0=6$ and $\mu_0'=\frac{24}{P_{-24}-1}<1$ by Remark~\ref{rem mu0}.
By Proposition~\ref{np2}(1) (with $t=6.9$), we can take $m_1=30$. Then by Theorem~\ref{thm bc}(2), $\varphi_{-m}$ is birational for all $m\geq56$.

From now on, we assume that $P_{-4}\geq2$.
Note that $-K_X^3\geq\frac{1}{30}$ by Proposition~\ref{known}(3). By Corollary~\ref{np2 cor} (with $l=1, t=3.6$), $P_{-16}-1>16$. If $|-16K_X|$ and $|-4K_X|$ are not composed with the same pencil, then take $m_1=16$ and $\mu_0'=m_0=4$. Then by Theorem~\ref{thm bc}(2), $\varphi_{-m}$ is birational for all $m\geq46$. 

In the following discussions, we assume that $|-16K_X|$ and $|-4K_X|$ are composed with the same pencil. We can always take $m_0=4$ and $\mu_0'=\frac{16}{P_{-16}-1}<1$ by Remark~\ref{rem mu0}.

\medskip

{\bf Case 1}. $r_{\max}\leq6$ or $r_{\max}\in \{10, 12\}$.

If $r_{\max}\leq6$, then $r_X\leq60$. If $r_{\max}=10$ (resp. $r_{\max}=12$), then $r_X\leq 210$ (resp. $r_X\leq 84$) by \cite[Page 107]{CJ16}.
Then by Corollary~\ref{usage of thm bc 2}(2), $\varphi_{-m}$ is birational for all $m\geq 52$.

\medskip

{\bf Case 2}. $r_{\max}=7$.

If $r_{\max}=7$, then $r_X$ divides $\lcm(2, 3, 4, 5, 6, 7)=420$.
Hence either $r_X=420$ or
$r_X\leq210$. 
If $r_X=420$, then $\{4, 5, 7\}\subset\mathcal{R}_X$ and one element of $\{3, 6\}$ is in $\mathcal{R}_X$.
Suppose that $$B_X=\{(b_1, r_1), \dots, (b_k, r_k), (1, r), (1, 4), (a_5, 5), (a_7, 7)\}$$ where $r\in \{3, 6\}$, $a_5\leq 2$, $a_7\leq 3$.
Then $\sigma(B_X)\geq 11$ implies that $\sum_{i=1}^kb_i\geq 4$. Lemma~\ref{3/2bi} implies that 
\begin{align}
\gamma(B_X)\leq 24-(7-\frac{1}{7}+5-\frac{1}{5}+4-\frac{1}{4}+3-\frac{1}{3}+4\times \frac{3}{2})<0, \label{use 3/2}
\end{align}
a contradiction.
Hence $r_X\leq 210$. Then by Corollary~\ref{usage of thm bc 2}(2), $\varphi_{-m}$ is birational for all $m\geq43$.

\medskip

{\bf Case 3}. $r_{\max}=8$.

If $r_{\max}=8$, then we claim that $r_X\leq168$. 
In fact, if $r_X>168$, then $\{5, 7\}\subset \mathcal{R}_X$. Suppose that $$B_X=\{(b_1, r_1), \dots, (b_k, r_k), (a_5, 5), (a_7, 7), (a_8, 8)\}$$ where $a_5\leq 2$, $a_7\leq 3$, $a_8\leq 3$.
Then $\sigma(B_X)\geq 11$ implies that $\sum_{i=1}^kb_i\geq 3$. Similar to \eqref{use 3/2}, Lemma~\ref{3/2bi} implies that $\gamma(B_X)<0$, a contradiction. Hence $r_X\leq168$.

Then by Corollary~\ref{usage of thm bc 2}(2), $\varphi_{-m}$ is birational for all $m\geq43$.

\medskip

{\bf Case 4}. $r_{\max}=9$.

If $r_{\max}=9$, then we claim that $r_X\leq252$ or $B_X=\{2\times(1, 2), (2, 5), (3, 7),$ $(4, 9)\}$. 

If $7$ and $8$ are not in $\mathcal{R}_X$, then $r_X\leq 180$. 

If $8\in \mathcal{R}_X$, then as $\{2, 8, 9\}\subset \mathcal{R}_X$, we know that $6$ and $7$ are not in $\mathcal{R}_X$ by $\gamma(B_X)\geq 0$. If $5\not \in \mathcal{R}_X$, then $r_X= 72$. If $5 \in \mathcal{R}_X$, then $ \mathcal{R}_X=\{2, 5, 8, 9\}$ as $\gamma(B_X)\geq 0$, but in this case $\sigma(B_X)\leq 10$, a contradiction.

If $7\in \mathcal{R}_X$, then as $\{2, 7, 9\}\subset \mathcal{R}_X$, we know that at most one element of $\{4, 5, 6\}$ is in $\mathcal{R}_X$ by $\gamma(B_X)\geq 0$. If $5\not \in \mathcal{R}_X$, then $r_X\leq 252$. If $5 \in \mathcal{R}_X$, then by $\sigma(B_X)\geq 11$ and $\gamma(B_X)\geq 0$, it is not hard to check that the only possible basket is $B_X=\{2\times(1, 2), (2, 5), (3, 7), (4, 9)\}$. This concludes the claim.

If $r_X\leq252$, then by Corollary~\ref{usage of thm bc 2}(2), $\varphi_{-m}$ is birational for all $m\geq50$.

If $B_X=\{2\times(1, 2), (2, 5), (3, 7), (4, 9)\}$, then $-K_X^3=\frac{43}{315}$. By Proposition~\ref{np2}(2) (with $t=7.6$), we can take $m_1=23$. Then by Theorem~\ref{thm bc}(2), $\varphi_{-m}$ is birational for all $m\geq41$.

\medskip

{\bf Case 5}. $r_{\max}=11$.

If $r_{\max}=11$, then we claim that $r_X\leq 264$ or $B_X=\{3\times(1, 2), (1, 3), (2, 5),$ $(5, 11)\}$ or $B_X=\{2\times(1, 2), (1, 3), (3, 7), (5, 11)\}$.

As $\{2, 11\}\subset \mathcal{R}_X$, we know that at most one element of $\{6, 7, 8, 9, 10\}$ is in $\mathcal{R}_X$ by $\gamma(B_X)\geq 0$.
If $10\in \mathcal{R}_X$, then $r_X= 110$ by \cite[Page 107]{CJ16}. If $9\in \mathcal{R}_X$ or $8\in \mathcal{R}_X$, then $r_X\leq 264$ by \cite[Page 107]{CJ16}. 
If $7\in \mathcal{R}_X$, then $5\not \in\mathcal{R}_X$ by $\gamma(B_X)\geq 0$. 
So either $r_X=154$ or at least one element of $\{3, 4\}$ is in $\mathcal{R}_X$. For the latter case, it is not hard to check that the only basket satisfying $\sigma(B_X)\geq 11$ and $\gamma(B_X)\geq 0$ is $B_X=\{2\times(1, 2), (1, 3), (3, 7), (5, 11)\}$.
If $6\in \mathcal{R}_X$, then we get a contradiction by Lemma~\ref{3/2bi} as \eqref{use 3/2}.

If none element of $\{6, 7, 8, 9, 10\}$ is in $\mathcal{R}_X$, then $r_X$ divides $660$ and $r_X<660$ by \cite[Page 107]{CJ16}. So either $r_X\leq 220$ or $r_X=330$. Moreover, if $r_X=330$, then $\{2, 3, 5, 11\}\subset \mathcal{R}_X$ and $4\not \in \mathcal{R}_X$, and it is not hard to check that the only basket satisfying $\sigma(B_X)\geq 11$ and $\gamma(B_X)\geq 0$ is $B_X=\{3\times(1, 2), (1, 3), (2, 5), (5, 11)\}$. This concludes the claim.

If $r_X\leq264$, then by Corollary~\ref{usage of thm bc 2}(2), $\varphi_{-m}$ is birational for all $m\geq56$.

If $B_X=\{3\times(1, 2), (1, 3), (2, 5), (5, 11)\}$, then $-K_X^3=\frac{31}{330}$. By Proposition~\ref{np2}(2) (with $t=7.6$), we can take $m_1=28$. Then by Theorem~\ref{thm bc}(2), $\varphi_{-m}$ is birational for all $m\geq50$.

If $B_X=\{2\times(1, 2), (1, 3), (3, 7), (5, 11)\}$, then $-K_X^3=\frac{50}{462}$. By Proposition~\ref{np2}(2) (with $t=7$), we can take $m_1=26$. Then by Theorem~\ref{thm bc}(2), $\varphi_{-m}$ is birational for all $m\geq48$.

\medskip

{\bf Case 6}. $r_{\max}=13$.

If $r_{\max}=13$, then $r_X\leq390$ or $r_X=546$ by \cite[Page~107]{CJ16}. 

 If $r_X=546$, then again by \cite[Page~107]{CJ16}, $B_X=\{(1, 2), (1, 3), (3, 7), (6, 13)\}$ and $-K_X^3=\frac{61}{546}$. By Proposition~\ref{np2}(2) (with $t=6$), we can take $m_1=26$. Then by Theorem~\ref{thm bc}(2), $\varphi_{-m}$ is birational for all $m\geq52$.

If $r_X\leq390$ and $-K_X^3\geq\frac{1}{12}$, then by Proposition~\ref{np2}(2) (with $t=6.9$), we can take $m_1=30$. Then by Theorem~\ref{thm bc}(2), $\varphi_{-m}$ is birational for all $m\geq56$. 

If $r_X\leq390$ and $-K_X^3<\frac{1}{12}$, then by Lemma~\ref{1/12}, 
$B_X$ is dominated by $\{8\times(1, 2), 3\times(1, 3)\}$. As $r_{\max}=13$, this implies that $B_X$ is dominated by either $\{3\times(1, 2), (6, 13), 2\times(1, 3)\}$ or $\{6\times(1, 2), (5, 13)\}$.
So we get the following possibilities of $B_X$ by $\gamma(B_X)\geq 0$:
\begin{align*}
 {}& \{3\times(1, 2), (6, 13), 2\times(1, 3)\} &{}-K^3=5/78;\\
 {}& \{2\times(1, 2), (6, 13), (2, 5), (1, 3)\}&{}-K^3=19/195;\\
 {}& \{2\times(1, 2), (6, 13), (3, 8)\}&{}-K^3=11/104;\\
 {}& \{(1, 2), (6, 13), (3, 7), (1, 3)\}&{}-K^3=61/546;\\
 {}& \{6\times(1, 2), (5, 13)\}&{}-K^3= 1/13.
\end{align*}
In the above list, only the first and the last have $-K_X^3<\frac{1}{12}$. In particular, in these cases, $r_X\leq 78$. Then by Corollary~\ref{usage of thm bc 2}(2), $\varphi_{-m}$ is birational for all $m\geq55$.

Combining all above cases, we have proved the theorem.
\end{proof}

\subsection{The case $P_{-1}>0$}

\begin{lem}\label{P-4}
Let $X$ be a terminal weak $\bQ$-Fano $3$-fold. If $P_{-1}>0$ and $r_{\max}\geq16$, then $P_{-4}\geq2$. 
\end{lem}

\begin{proof}
If $P_{-4}=1$, then $P_{-2}=P_{-3}=1$. Since $r_{\max}\geq16$, by the classification in \cite[Subsubcase II-4f of Proof of Theorem~4.4]{CC08}, $B_X$ is dominated by $\{2\times(1, 2), 2\times(1, 3), (1, s_1), (1, s_2)\}$ with $s_2\geq s_1\geq5$, which means that $s_1+s_2=r_{\max}\geq16$. But in this case
$2\times \frac{3}{2}+2\times \frac{8}{3}+16-\frac{1}{16}>24$, contradicting \eqref{gamma>0} and \cite[Lemma~3.1]{CC08}. So $P_{-4}\geq2$.
\end{proof}

\begin{thm}\label{case4}
Let $X$ be a terminal weak $\bQ$-Fano $3$-fold. If $P_{-1}>0$ and $r_{\max } \geq 14$, then $\varphi_{-m}$ is birational for all $m\geq52$.
\end{thm}

\begin{proof}
Keep the setting in Notation~\ref{set up}. We always take $\nu_0=1$.

If $14\leq r_{\max} \leq 15$, then $r_X\leq 210$ by \cite[Page~104]{CJ16}. By Proposition~\ref{known}(2), we can take $\mu_{0}'=m_0=8$. Then 
by Corollary~\ref{usage of thm bc 2}(2), $\varphi_{-m}$ is birational for all $m\geq 52$.

If $r_{\max}=24$, then $B_X=\{(b, 24)\}$ with $b\in\{1, 5, 7, 11\}$.
If $P_{-1}=1$, then $b=\sigma(B_X)=5+P_{-2}\geq 6$ by \eqref{sigma}, hence $b\geq 7$ and $P_{-2}\geq 2$. By \eqref{volume}, we have $-K_X^3\geq\frac{23}{24}$. 
Similarly, if $P_{-1}=2$, then $P_{-2}\geq2P_{-1}-1=3$, and thus $b\geq5$. Hence by \eqref{volume}, $-K_X^3\geq\frac{47}{24}$. 
If $P_{-1}\geq3$, then by \eqref{volume}, $-K_X^3\geq b-\frac{b^2}{24}\geq\frac{23}{24}$. In summary, $-K_X^3\geq\frac{23}{24}$ and $P_{-2}\geq2$. We can take $\mu_{0}'=m_0=2$. By Proposition~\ref{np2}(2) (with $t=1$), we can take $m_1=9$. Then by Theorem~\ref{thm bc}(1), $\varphi_{-m}$ is birational for all $m\geq33$.

In the following, we consider $16\leq r_{\max}\leq 23$. 
By Lemma~\ref{P-4}, we always take $\mu_{0}'=m_0=4$ and $\nu_0=1$.

If $r_{\max}=23$, then $B_X=\{(b, 23)\}$ with $1\leq b\leq11$.
If $P_{-1}=1$, then $b=5+P_{-2}\geq6$ and thus by \eqref{volume} $-K_X^3\geq\frac{10}{23}$. 
If $P_{-1}=2$, then $b=P_{-2}\geq2P_{-1}-1=3$, hence by \eqref{volume} $-K_X^3\geq\frac{14}{23}$. 
If $P_{-1}\geq3$, then $-K_X^3\geq b-\frac{b^2}{23}\geq\frac{22}{23}$.
In summary, $-K_X^3\geq\frac{10}{23}$. 
By Proposition~\ref{np2}(1) (with $t=\frac{36}{23}$), we can take $m_1=12$. Then by Theorem~\ref{thm bc}(1), $\varphi_{-m}$ is birational for all $m\geq48$. 

If $20\leq r_{\max} \leq22$, then by \eqref{gamma>0} we have $r_X\leq 60$. Then 
by Corollary~\ref{usage of thm bc 2}(2), $\varphi_{-m}$ is birational for all $m\geq 50$.

If $18\leq r_{\max} \leq 19$, then $r_X\leq 190$ by \cite[Page~104]{CJ16}. Then 
by Corollary~\ref{usage of thm bc 2}(2), $\varphi_{-m}$ is birational for all $m\geq 52$.

If $16\leq r_{\max} \leq 17$, then $r_X\leq 240$ by \cite[Page~104]{CJ16}. Then 
by Corollary~\ref{usage of thm bc 2}(2), $\varphi_{-m}$ is birational for all $m\geq 52$.
\end{proof}

\begin{thm}\label{case5}
Let $X$ be a terminal weak $\bQ$-Fano $3$-fold. If $P_{-1}>0$ and $r_{\max }\leq 13$, then $\varphi_{-m}$ is birational for all $m\geq58$.
\end{thm}

\begin{proof}
Keep the setting in Notation~\ref{set up}.
By Theorem~\ref{case0} and Proposition~\ref{known}(1), we may assume that $r_X\leq660$.
By Proposition~\ref{known}(2), we can take $m_0=8$ and $\nu_0=1$. We take $\mu_0'=8$ unless stated otherwise.

\medskip

{\bf Case 1.} $r_{\max}\leq8$.

If $r_{\max}\leq8$, then $r_{X}$ divides $\lcm(8, 7, 6, 5)=840$. As $r_{X}\leq660$, $r_X=420$ or $r_X\leq280$. 

If $r_X=420$, then by Proposition~\ref{np2}(1) (with $t=19.5$, $-K_X^3\geq\frac{1}{330}$), we can take $m_1=52$. By Lemma~\ref{N_0}, $N_0\geq\lceil\frac{420}{8m_{1}}\rceil=2$. Then by Corollary~\ref{usage of thm bc 2}(1), $\varphi_{-m}$ is birational for all $m\geq48$. 

If $r_X\leq280$, then by Corollary~\ref{usage of thm bc 2}(1), $\varphi_{-m}$ is birational for all $m\geq55$.

\medskip

{\bf Case 2.} $r_{\max}=9$.

If $r_{\max}=9$, then $r_X$ divides $2520$. As $r_X\leq 660$ and $9$ divides $r_X$, we have
$r_X\leq360$ or $r_X\in\{504, 630\}$. 

If $r_X\leq360$, then by Corollary~\ref{np2 cor} (with $l=4, t=10$, $-K_X^3\geq\frac{1}{330}$), we have $P_{-30}-1>\frac{30}{4}$. If $|-30K_X|$ and $|-8K_X|$ are composed with the same pencil, then take $\mu_0'=\frac{30}{P_{-30}-1}<4$ by Remark~\ref{rem mu0}. Then by Corollary~\ref{usage of thm bc 2}(1), $\varphi_{-m}$ is birational for all $m\geq57$. If $|-30K_X|$ and $|-8K_X|$ are not composed with the same pencil, then take $m_1=30$ and $\mu_0'=m_0=8$. Then by Theorem~\ref{thm bc}(3), $\varphi_{-m}$ is birational for all $m\geq56$.

If $r_X=630$, then $-K_X^3\geq\frac{1}{315}$ (note that $-K_X^3\geq\frac{1}{330}$ and $r_X(-K_X^3)$ is an integer). By Corollary~\ref{np2 cor} (with $l=2.8, t=11$), we have $P_{-33}-1>\frac{33}{2.8}$. 
If $|-33K_X|$ and $|-8K_X|$ are composed with the same pencil, then take $\mu_0'=\frac{33}{P_{-33}-1}<2.8$ by Remark~\ref{rem mu0}. By Proposition~\ref{np2}(1) (with $t=20$, $-K_X^3\geq\frac{1}{315}$), we can take $m_1=63$. By Lemma~\ref{N_0}, $N_0\geq\lceil\frac{630}{9m_{1}}\rceil=2$.
Then by Corollary~\ref{usage of thm bc 2}(1),
 $\varphi_{-m}$ is birational for all $m\geq52$. 
If $|-33K_X|$ and $|-8K_X|$ are not composed with the same pencil, then take $m_1=33$ and $\mu_0'=m_0=8$. By Lemma~\ref{N_0}, $N_0\geq\lceil\frac{630}{9m_{1}}\rceil=3$. Then by Corollary~\ref{usage of thm bc 2}(1),
 $\varphi_{-m}$ is birational for all $m\geq48$. 

If $r_X=504$, then $\mathcal{R}_X=\{9, 8, 7\}$ by \eqref{gamma>0}. Write $B_X=\{(a, 7), (b, 8), (c, 9)\}$, where $a\leq3, b\in\{1, 3\}$ and $c\in\{1, 2, 4\}$. If $P_{-1}\geq 2$, then by \eqref{volume}, 
\begin{align}
 -K_X^{3}{}&=2P_{-1}+\frac{a(7-a)}{7}+\frac{b(8-b)}{8}+\frac{c(9-c)}{9}-6\notag\\
 {}&\geq \frac{6}{7}+\frac{7}{8}+\frac{8}{9}-2>0.6\label{similar1}.
\end{align}
If $P_{-1}=1$, then by \eqref{volume} and \eqref{sigma},
\begin{align}
 -K_X^{3}{}&=\frac{a(7-a)}{7}+\frac{b(8-b)}{8}+\frac{c(9-c)}{9}-4>0,\notag\\
 \sigma(B_X){}&=a+b+c=5+P_{-2}\geq6.\label{similar1+}
\end{align}
So it is easy to check that $-K_X^3\geq \frac{73}{504}$ by 
considering all possible values of $(a, b, c)$.
By Proposition~\ref{np2}(2) (with $t=7.3$, $-K_X^3\geq\frac{73}{504}$), we can take $m_1=22$. By Lemma~\ref{N_0}, $N_0\geq \lceil\frac{504}{9m_{1}}\rceil=3$. Then by Corollary~\ref{usage of thm bc 2}(1), $\varphi_{-m}$ is birational for all $m\geq44$.

\medskip

{\bf Case 3.} $r_{\max}=10$.

If $r_{\max}=10$, then we claim that $r_X\leq210$ or $r_X=420$. 

By \eqref{gamma>0}, at most one element of $\{7, 8, 9\}$ is in $\mathcal{R}_X$. If $7\not \in \mathcal{R}_X$, then $r_X$ divides either $120=\lcm(10, 8, 60)$ or $180=\lcm(10, 9, 60)$. 
If $7 \in \mathcal{R}_X$, then $r_X$ divides $420=\lcm(10, 7, 60)$, so either $r_X\leq 210$ or $r_X=420$. This concludes the claim.

If $r_X\leq210$, then by Corollary~\ref{usage of thm bc 2}(1), $\varphi_{-m}$ is birational for all $m\geq48$.

If $r_X=420$, then by \eqref{gamma>0}, $\mathcal{R}_X=\{10, 7, 4, 3\}$. 
Write $B_X=\{(1, 3), (1, 4),$ $(a, 7), (b, 10)\}$, where $a\leq3$ and $b\in\{1, 3\}$. If $P_{-1}\geq 2$, then by \eqref{volume}, 
\begin{align}
 -K_X^{3}{}&=2P_{-1}+\frac{2}{3}+\frac{3}{4}+\frac{a(7-a)}{7}+\frac{b(10-b)}{10}-6\notag\\
 {}&\geq \frac{2}{3}+\frac{3}{4}+\frac{6}{7}+\frac{9}{10}-2>1. \label{similar2}
\end{align}
If $P_{-1}=1$, then by \eqref{volume} and \eqref{sigma},
\begin{align}
 -K_X^{3}{}&=\frac{2}{3}+\frac{3}{4}+\frac{a(7-a)}{7}+\frac{b(10-b)}{10}-4>0,\notag\\
 \sigma(B_X){}&=2+a+b=5+P_{-2}\geq6.\label{similar2+}
\end{align}
So it is easy to check that $-K_X^3\geq \frac{13}{420}$ by 
considering all possible values of $(a, b)$.
By Corollary~\ref{np2 cor} (with $l=1$, $t=4.8$), we have $P_{-16}-1>16$.
If $|-16K_X|$ and $|-8K_X|$ are composed with the same pencil, then take $\mu_0'=\frac{16}{P_{-16}-1}<1$ by Remark~\ref{rem mu0}. Then by Corollary~\ref{usage of thm bc 2}(1), $\varphi_{-m}$ is birational for all $m\geq58$.
If $|-16K_X|$ and $|-8K_X|$ are not composed with the same pencil, then take $m_1=16$ and $\mu_0'=m_0=8$. Then by Theorem~\ref{thm bc}(3), $\varphi_{-m}$ is birational for all $m\geq44$.

\medskip

{\bf Case 4.} $r_{\max}=11$.

If $r_{\max}=11$, then we claim that $r_X\leq330$ or $r_X\in\{385, 396, 440, 462, 660\}$.

By \eqref{gamma>0}, at most one element of $\{7, 8, 9, 10\}$ is in $\mathcal{R}_X$.
So $r_X$ divides one element of $\{1980, 1320, 4620\}$. As $r_X\leq 660$ and $11$ divides $r_X$, it is clear that $r_X\leq330$ or $r_X\in\{385, 396, 440, 462, 495, 660\}$.
Moreover, if $r_X=495$, 
then $\{11, 9, 5\}\subset \mathcal{R}_X$, which contradicts \eqref{gamma>0}. This concludes the claim.

If $r_X\leq330$, then by Corollary~\ref{np2 cor} (with $l=7.6$, $t=7.6$, $-K_X^3\geq\frac{1}{330}$), $P_{-28}-1>\frac{28}{7.6}$. If $|-28K_X|$ and $|-8K_X|$ are composed with the same pencil, then take $\mu_0'=\frac{28}{P_{-28}-1}<7.6$ by Remark~\ref{rem mu0}. Then by Corollary~\ref{usage of thm bc 2}(1), $\varphi_{-m}$ is birational for all $m\geq58$. If $|-28K_X|$ and $|-8K_X|$ are not composed with the same pencil, then take $m_1=28$ and $\mu_0'=m_0=8$. Then by Theorem~\ref{thm bc}(3), $\varphi_{-m}$ is birational for all $m\geq58$.

If $r_X=385$ (resp. $396$), then $-K_X^3\geq\frac{2}{385}$ (resp. $\geq\frac{2}{396}$). By Corollary~\ref{np2 cor} (with $l=1.5, t=9$), $P_{-33}-1>\frac{33}{1.5}$. If $|-33K_X|$ and $|-8K_X|$ are composed with the same pencil, then take $\mu_0'=\frac{33}{P_{-33}-1}<1.5$ by Remark~\ref{rem mu0}.
Then by Corollary~\ref{usage of thm bc 2}(1), $\varphi_{-m}$ is birational for all $m\geq56$ (resp. $\geq57$). 
If $|-33K_X|$ and $|-8K_X|$ are not composed with the same pencil, then take $m_1=33$ and $\mu_0'=m_0=8$. By Lemma~\ref{N_0}, $N_0\geq \lceil\frac{385}{11m_{1}}\rceil= 2$. Then by Corollary~\ref{usage of thm bc 2}(2), 
$\varphi_{-m}$ is birational for all $m\geq47$ (resp. $\geq48$).

We claim that if $r_X\in\{440, 462, 660\}$, then $-K_X^3\geq\frac{74}{462}$ or $B_X=\{(1, 2), (2, 5), (1,$
$ 3), (1, 4), (1, 11)\}$
with $-K_X^3=\frac{17}{660}$.

If $r_X=440$, then by \eqref{gamma>0}, $\mathcal{R}_X=\{11, 8, 5\}$. 
Arguing similarly as \eqref{similar1}, we get $-K_X^3>0.5$ when $P_{-1}\geq2$. Arguing similarly as \eqref{similar1+}, we get constrains for $B_X$ when $P_{-1}=1$. By \eqref{volume} and considering all the possible baskets when $P_{-1}=1$, we can see that $-K_X^3\geq \frac{97}{440}$.

If $r_X=462$, then by \eqref{gamma>0}, $\mathcal{R}_X=\{11, 7, 6\}$ or $\{11, 7, 3, 2\}$ or $\{11, 7, 3, 2, 2\}$. 
Arguing similarly as \eqref{similar2}, we get $-K_X^3>0.5$ or $-K_X^3>0.9$ or $-K_X^3>1$ when $P_{-1}\geq2$. Arguing similarly as \eqref{similar2+}, we get constrains for $B_X$ when $P_{-1}=1$. By \eqref{volume} and considering all the possible baskets when $P_{-1}=1$, we can see that
$-K_X^3\geq \frac{85}{462}$ or 
$-K_X^3\geq \frac{95}{462}$ or 
$-K_X^3\geq \frac{74}{462}$ unless
$B_X=\{2\times(1, 2), (1, 3), (2, 7), (1, 11)\}$
with $-K_X^3=\frac{1}{231}$. But the last basket has $P_{-5}=0$, which is absurd.

If $r_X=660$, then by \eqref{gamma>0}, $\mathcal{R}_X=\{11, 5, 4, 3\}$ or $\{11, 5, 4, 3, 2\}$. 
Arguing similarly as \eqref{similar2}, we get $-K_X^3>1$ or $-K_X^3>1.5$ when $P_{-1}\geq2$. Arguing similarly as \eqref{similar2+}, we get constrains for $B_X$ when $P_{-1}=1$. By \eqref{volume} and considering all the possible baskets when $P_{-1}=1$, we can see that
$-K_X^3\geq \frac{167}{660}$ or 
$-K_X^3\geq \frac{233}{660}$ unless
$B_X=\{(1, 2), (2, 5), (1, 3), (1, 4), (1, 11)\}$
with $-K_X^3=\frac{17}{660}$. 

To summarize, we conclude the claim that $-K_X^3\geq\frac{74}{462}$ or $B_X=\{(1, 2), (2,$ $5), (1, 3), (1,$ $4), (1, 11)\}$
with $-K_X^3=\frac{17}{660}$.

If $-K_X^3\geq\frac{74}{462}$, then 
by Proposition~\ref{np2}(2) (with $t=5.7$), 
we can take $m_1=21$. Then by Theorem~\ref{thm bc}(3), $\varphi_{-m}$ is birational for all $m\geq51$.

Now we consider the case $B_X=\{(1, 2), (2, 5), (1, 3), (1, 4), (1, 11)\}$ with $-K_X^3=\frac{17}{660}$. By Corollary~\ref{np2 cor} (with $l=1, t=4.8$), $P_{-18}-1>18$. If $|-18K_X|$ and $|-8K_X|$ are composed with the same pencil, then take $\mu_0'=\frac{18}{P_{-18}-1}<1$ by Remark~\ref{rem mu0}.
By Proposition~\ref{np2}(2) (with $t=14.4$), we can take $m_1=53$. 
By Lemma~\ref{N_0}, $N_0\geq \lceil\frac{660}{11m_{1}}\rceil=2$. Then by Corollary~\ref{usage of thm bc 2}(1), $\varphi_{-m}$ is birational for all $m\geq52$.
If $|-18K_X|$ and $|-8K_X|$ are not composed with the same pencil, then take $m_1=18$ and $\mu_0'=m_0=8$. By Lemma~\ref{N_0}, $N_0\geq\lceil\frac{660}{11m_{1}}\rceil=4$. Then by Corollary~\ref{usage of thm bc 2}(2), $\varphi_{-m}$ is birational for all $m\geq45$.

\medskip

{\bf Case 5.} $r_{\max}=12$.

If $r_{\max}=12$, then we claim that $r_X\leq132$ or $r_X=420$.

By \eqref{gamma>0}, at most one element of $\{6, 7, 8, 9, 10, 11\}$ is in $\mathcal{R}_X$.
So $r_X$ divides one element of $\{120, 180, 420, 660\}$. Recall that $12$ divides $r_X$, it is clear that $r_X\leq132$ or $r_X\in\{180, 420, 660\}$.
Moreover, if $r_X\in \{180, 660\}$, then $\{12, 9, 5\}\subset \mathcal{R}_X$ or $\{12, 11, 5\}\subset \mathcal{R}_X$, which contradicts \eqref{gamma>0}. This concludes the claim.

If $r_X\leq132$, then by Corollary~\ref{usage of thm bc 2}(2), $\varphi_{-m}$ is birational for all $m\geq43$.

If $r_X=420$, then by \eqref{gamma>0}, $\mathcal{R}_X=\{12, 7, 5\}$. Arguing similarly as \eqref{similar1}, we get $-K_X^3\geq\frac{241}{420}$ when $P_{-1}\geq2$. Arguing similarly as \eqref{similar1+}, we get constrains for $B_X$ when $P_{-1}=1$. By \eqref{volume} and considering all the possible baskets when $P_{-1}=1$, we can see that
$-K_X^3\geq \frac{241}{420}$. 
By Proposition~\ref{np2}(2) (with $t=2.75$), we can take $m_1=11$. Hence by Lemma~\ref{N_0}, $N_0\geq\lceil\frac{420}{12m_{1}}\rceil=4$. Then by Corollary~\ref{usage of thm bc 2}(2), $\varphi_{-m}$ is birational for all $m\geq 40$.

\medskip

{\bf Case 6.} $r_{\max}=13$.

If $r_{\max}=13$, then we claim that $r_X\leq364$ or $r_X=390$ or $r_X=546$.

By \eqref{gamma>0}, at most one element of $\{6, 7, 8, 9, 10, 11, 12\}$ is in $\mathcal{R}_X$.
So $r_X$ divides one element of $\{5460, 1560, 2340, 8580\}$. Recall that $13$ divides $r_X$ and $r_X\leq 660$, it is clear that $r_X\leq 364$ or $r_X\in\{390, 429, 455, 468,$ $520, 546, 572,$ $585\}$.
Moreover, if $r_X\in \{429, 455, 468,$ $520, 572, 585\}$, then we can see that $\mathcal{R}_X$ violates \eqref{gamma>0} by discussing the factors. For example, if $r_X=455$, then $\{13, 7, 5\}\subset \mathcal{R}_X$ which violates \eqref{gamma>0}. This concludes the claim.

If $P_{-4}=1$, then $P_{-k}=1$ for $1\leq k\leq 4$. By \cite[Subsubcase~II-4f of Proof of Theorem~4.4]{CC08}, $B_X$ is dominated by 
$$
\{2\times (1, 2), 2\times (1, 3), (1, s_1), (1, s_2)\}
$$
for some $s_2\geq s_1\geq 4$.
As $r_{\max}=13$, $(s_1, s_2)=(6, 7)$. Considering all possible packings, we get the following possibilities of $B_X$:
\begin{align*}
 {}& \{2\times (1, 2), 2\times (1, 3), (2, 13)\};\\
 {}& \{(1, 2), (2, 5), (1, 3), (2, 13)\};\\
 {}& \{(3, 7), (1, 3), (2, 13)\};\\
 {}& \{(1, 2), (3, 8), (2, 13)\};\\
 {}& \{2\times(2, 5), (2, 13)\}.
\end{align*}
Then $r_X\leq273$ or $B_X=\{(1, 2), (2, 5), (1, 3), (2, 13)\}$. 

If $r_X\leq273$, then by Corollary~\ref{usage of thm bc 2}(2), $\varphi_{-m}$ is birational for all $m\geq55$. 

If $B_X=\{(1, 2), (2, 5), (1, 3), (2, 13)\}$, then $-K_X^3=\frac{23}{390}$. By Corollary~\ref{np2 cor} (with $l=1, t=3$), $P_{-13}-1>13$. If $|-13K_X|$ and $|-8K_X|$ are composed with the same pencil, then take $\mu_0'=\frac{13}{P_{-13}-1}<1$ by Remark~\ref{rem mu0}. Then by Corollary~\ref{usage of thm bc 2}(1), $\varphi_{-m}$ is birational for all $m\geq56$. If $|-13K_X|$ and $|-8K_X|$ are not composed with the same pencil, then take $m_1=13$ and $\mu_0'=m_0=8$. Then by Lemma~\ref{N_0}, $N_0\geq\lceil\frac{390}{13m_{1}}\rceil=3$. Then by Corollary~\ref{usage of thm bc 2}(2), $\varphi_{-m}$ is birational for all $m\geq44$.

So from now on, we assume that $P_{-4}\geq 2$ and take $m_0=4$. We take $\mu_0'=4$ unless stated otherwise.

If $r_X\leq364$, then by Corollary~\ref{usage of thm bc 2}(1), $\varphi_{-m}$ is birational for all $m\geq57$.

If $r_X=546$, then by \eqref{gamma>0}, $\mathcal{R}_X=\{13, 7, 3, 2\}$. 
Arguing similarly as \eqref{similar2}, we get $-K_X^3>0.9$ when $P_{-1}\geq2$. Arguing similarly as \eqref{similar2+}, we get constrains for $B_X$ when $P_{-1}=1$. By \eqref{volume} and considering all the possible baskets when $P_{-1}=1$, we can see that
$-K_X^3\geq \frac{157}{546}$.
 By Proposition~\ref{np2}(2) (with $t=3.6$), we can take $m_1=16$. Hence by Lemma~\ref{N_0}, $N_0\geq\lceil\frac{546}{13m_{1}}\rceil=3$. Then by Corollary~\ref{usage of thm bc 2}(2), $\varphi_{-m}$ is birational for all $m\geq 44$.

If $r_X=390$, then by \eqref{gamma>0}, $\mathcal{R}_X=\{13, 6, 5\}$ or $\{13, 5, 3, 2\}$ or $\{13, 5, 3, 2, 2\}$. 
Arguing similarly as \eqref{similar2}, we get $-K_X^3>0.5$ or $-K_X^3>0.8$ or $-K_X^3>1$ when $P_{-1}\geq2$. Arguing similarly as \eqref{similar2+}, we get constrains for $B_X$ when $P_{-1}=1$. By \eqref{volume} and considering all the possible baskets when $P_{-1}=1$, we can see that
$-K_X^3\geq \frac{133}{390}$ or 
$-K_X^3\geq \frac{23}{390}$ or 
$-K_X^3\geq \frac{62}{390}$. 
By Corollary~\ref{np2 cor} (with $l=1, t=3$, $-K_X^3\geq\frac{23}{390}$), $P_{-13}-1>13$. If $|-13K_X|$ and $|-4K_X|$ are composed with the same pencil, then take $\mu_0'=\frac{13}{P_{-13}-1}<1$ by Remark~\ref{rem mu0}. Then by Corollary~\ref{usage of thm bc 2}(1), $\varphi_{-m}$ is birational for all $m\geq56$. If $|-13K_X|$ and $|-4K_X|$ are not composed with the same pencil, then take $m_1=13$ and $\mu_0'=m_0=4$. Then by Lemma~\ref{N_0}, $N_0\geq\lceil\frac{390}{13m_{1}}\rceil=3$. Then by Corollary~\ref{usage of thm bc 2}(2), $\varphi_{-m}$ is birational for all $m\geq40$.

Combining all above cases, we have proved the theorem.
\end{proof}

\subsection{Proof of main results}

\begin{proof}[Proof of Theorem~\ref{thm1}]
It follows from 
Theorems~\ref{case1}, \ref{case2}, \ref{case3}, \ref{case4}, \ref{case5}.
\end{proof}

\begin{proof}[Proof of Theorem~\ref{thm2}]
From the proof of Theorem~\ref{thm1}, $\varphi_{-58}$ may not be birational only if $B_X=\{(1, 2), 2\times (1, 3), (8, 17)\}$, $P_{-1}=0$, and $|-24K_X|$ is composed with a pencil. In this case, $r_X(-K_X^3)=7$ and $P_{-24}=169=7\times 24+1$ by \eqref{RR for P}. But this contradicts Proposition~\ref{np1 equality}.
\end{proof}

\section{Appendix}

The possible baskets with $P_{-2}=0$ are the following (cf. \cite[Theorem~3.5]{CC08}).
{\footnotesize
\begin{longtable}{CLCCCCCCC}
\caption{Baskets with $P_{-2}=0$}\label{tabA}\\
\hline
\text{No.} & B_X & -K^3 & P_{-3} & P_{-4} & P_{-5} & P_{-6} & P_{-7} & P_{-8}\\
\hline
\endfirsthead
\multicolumn{3}{l}{{ {\bf \tablename\ \thetable{}} \textrm{-- continued from previous page}}}
\\
\hline 
\text{No.} & B_X & -K^3 & P_{-3} & P_{-4} & P_{-5} & P_{-6} & P_{-7} & P_{-8} \\
\endhead
\hline
\hline \multicolumn{3}{r}{{\textrm{Continued on next page}}} \\ \hline
\endfoot

\hline \hline
\endlastfoot

1& \{2\times(1, 2), 3\times(2, 5), (1, 3), (1, 4)\} & 1/60 & 0 & 0 & 1 & 1 & 1 & 2\\
2& \{5\times(1, 2), 2\times(1, 3), (2, 7), (1, 4)\} & 1/84 & 0 & 1 & 0 & 1 & 1 & 2\\
3& \{5\times(1, 2), 2\times(1, 3), (3, 11)\} & 1/66 & 0 & 1 & 0 & 1 & 1 & 2\\
4& \{5\times(1, 2), (1, 3), (3, 10), (1, 4)\} & 1/60 & 0 & 1 & 0 & 1 & 1 & 2\\
5& \{5\times(1, 2), (1, 3), 2\times(2, 7)\} & 1/42 & 0 & 1 & 0 & 1 & 2 & 3\\
6& \{4\times(1, 2), (2, 5), 2\times(1, 3), 2\times(1, 4)\} & 1/30 & 0 & 1 & 1 & 2 & 2 & 4\\
7& \{3\times(1, 2), (2, 5), 5 \times (1, 3)\} & 1/30 & 1 & 1 & 1 & 3 & 3 & 4\\
8& \{2 \times (1, 2), (3, 7), 5 \times (1, 3)\} & 1/21 & 1 & 1 & 1 & 3 & 4 & 5\\
9& \{(1, 2), (4, 9), 5 \times (1, 3)\} & 1/18 & 1 & 1 & 1 & 3 & 4 & 5\\
10& \{3 \times (1, 2), (3, 8), 4 \times (1, 3)\} & 1/24 & 1 & 1 & 1 & 3 & 3 & 5\\
11& \{3 \times (1, 2), (4, 11), 3 \times (1, 3)\} & 1/22 & 1 & 1 & 1 & 3 & 3 & 5\\
12& \{3 \times (1, 2), (5, 14), 2 \times (1, 3)\} & 1/21 & 1 & 1 & 1 & 3 & 3 & 5\\
13& \{2 \times (1, 2), 2 \times (2, 5), 4 \times (1, 3)\} & 1/15 & 1 & 1 & 2 & 4 & 5 & 7\\
14& \{(1, 2), (3, 7), (2, 5), 4 \times (1, 3)\} & 17/210 & 1 & 1 & 2 & 4 & 6 & 8\\
15& \{2 \times (1, 2), (2, 5), (3, 8), 3 \times (1, 3)\} & 3/40 & 1 & 1 & 2 & 4 & 5 & 8\\
16& \{2 \times (1, 2), (5, 13), 3 \times (1, 3)\} & 1/13 & 1 & 1 & 2 & 4 & 5 & 8\\
17& \{(1, 2), 3 \times (2, 5), 3 \times (1, 3)\} & 1/10 & 1 & 1 & 3 & 5 & 7 & 10\\
18& \{4 \times (1, 2), 5 \times (1, 3), (1, 4)\} & 1/12 & 1 & 2 & 2 & 5 & 6 & 9\\
19& \{4 \times (1, 2), 4 \times (1, 3), (2, 7)\} & 2/21 & 1 & 2 & 2 & 5 & 7 & 10\\
20& \{4 \times (1, 2), 3 \times (1, 3), (3, 10)\} & 1/10 & 1 & 2 & 2 & 5 & 7 & 10\\
21& \{3 \times (1, 2), (2, 5), 4 \times (1, 3), (1, 4)\} & 7/60 & 1 & 2 & 3 & 6 & 8 & 12\\
22& \{3 \times (1, 2), 7 \times (1, 3)\} & 1/6 & 2 & 3 & 4 & 9 & 12 & 17\\
23& \{2 \times (1, 2), (2, 5), 6 \times (1, 3)\} & 1/5 & 2 & 3 & 5 & 10 & 14 & 20\\
\hline
\end{longtable}
}

The possible baskets with $P_{-1}=0$ and $P_{-2}=P_{-4}=1$ are the following (cf. the second paragraph of \cite[Case~IV of Proof of Theorem~1.8]{CJ16}).
{\footnotesize
\begin{longtable}{CLC}
\caption{Baskets with $P_{-1}=0$ and $P_{-2}=P_{-4}=1$}\label{tabB}\\
\hline
\text{No.} & B_X & r_X \\
\hline
\endfirsthead
\multicolumn{3}{l}{{ {\bf \tablename\ \thetable{}} \textrm{-- continued from previous page}}}
\\
\hline 
\text{No.} & B_X & r_X \\
\endhead
\hline
\hline \multicolumn{3}{r}{{\textrm{Continued on next page}}} \\ \hline
\endfoot

\hline \hline
\endlastfoot

1& \{9\times(1, 2),(1, 3),(1, 7)\} & 42\\
2& \{8\times(1, 2),(2, 5),(1, 7)\} & 70\\
3& \{8\times(1, 2),(2, 5),(1, 6)\} & 30\\
4& \{7\times(1, 2),(3, 7),(1, 6)\} & 42\\
5& \{6\times(1, 2),(4, 9),(1, 6)\} & 18\\
6& \{7\times(1, 2),(3, 7),(1, 5)\} & 70\\
7& \{6\times(1, 2),(4, 9),(1, 5)\} & 90\\
8& \{5\times(1, 2),(5, 11),(1, 5)\} & 110\\
9& \{4\times(1, 2),(6, 13),(1, 5)\} & 130\\
\hline
\end{longtable}
}

\section*{Acknowledgments} 
The second named author would like to thank her advisor, Professor Meng Chen, for his support and encouragement. The authors are grateful to Meng Chen for discussions and suggestions. This work was supported by NSFC for Innovative Research Groups (Grant No.~12121001) and National Key Research and Development Program of China (Grant No.~2020YFA0713200). The first author is a member of LMNS, Fudan University. We would like to thank the referee for useful suggestions.

\end{document}